\documentclass[10pt]{amsart}
\usepackage{hyperref}
\usepackage{amsmath}
\usepackage{amssymb}
\usepackage{xypic}
\usepackage{yhmath}
\def\beqnn{\begin{eqnarray*}}\def\eeqnn{\end{eqnarray*}}

\newtheorem{theorem}{Theorem}[section]
\newtheorem{lemma}[theorem]{Lemma}

\theoremstyle{remark}
\newtheorem{remark}[theorem]{Remark}

\theoremstyle{definition}

\theoremstyle{problem}

\theoremstyle{conjecture}

\numberwithin{equation}{section}

\begin{document}

\begin{center}
\title[$p$-adic Hardy-Littlewood-P\'{o}lya-type operators]{Boundedness and norm of certain $p$-adic Hardy-Littlewood-P\'{o}lya-type operators}
\end{center}

\author{Jianjun Jin }
\address{School of Mathematics Sciences, Hefei University of Technology, Xuancheng Campus, Xuancheng 242000, P.R.China}
\email{jin@hfut.edu.cn, jinjjhb@163.com}
\author{Huabing Li}
\address{School of Mathematics Sciences, Hefei University of Technology, Xuancheng Campus, Xuancheng 242000, P.R.China}
\email{musicli121@163.com}
\thanks{The authors were supported by National Natural Science Foundation of China (Grant Nos. 11501157).}



\begin{abstract} 
In this paper, by introducing some parameters, we define and study certain $p$-adic Hardy-Littlewood-P\'{o}lya-type integral operators acting on $p$-adic weighted Lebesgue spaces. We completely characterize $L^{q}-L^{r}$ boundedness of these operators for all $(q, r)\in [1, \infty]\times[1, \infty]$. For some special cases, we obtain sharp norm estimates for the operators. These results are not only a complement to some previous results but also an extension of existing ones in the literature.\end{abstract}
\maketitle

\section{\bf {Introduction and main results} }

Let $q\geq 1$, we denote the conjugate of $q$ by $q'$,  i.e. $\frac{1}{q}+\frac{1}{q'}=1$. In particular, $q'=\infty$ when $q=1$. Let $\mathbb{R}_{+}=(0,+\infty)$.  Let $\mathcal{M}(\mathbb{R}_{+})$ be the class of all real-valued measurable functions on $\mathbb{R}_{+}$ and let $L^q(\mathbb{R}_{+})$ be the usual Lebesgue space on $\mathbb{R}_{+}$, i.e.,
\begin{equation*}
L^q(\mathbb{R}_{+})=\{f\in \mathcal{M}(\mathbb{R}_{+}): \|f\|_q=(\int_{\mathbb{R}_{+}}
|f(x)|^qdx)^{\frac{1}{q}}<\infty\}.
\end{equation*}
When $q>1$. If $f \in L^q(\mathbb{R}_{+}), g \in L^{q'}(\mathbb{R}_{+})$,  then we have the following famous Hardy-Littlewood-P\'{o}lya ({\bf{HLP}}) inequality
\begin{equation}\label{HLP}
\Bigg|\int_{\mathbb{R}_{+}}
\int_{\mathbb{R}_{+}}\frac{f(x)g(y)}{\max\{x, y\}}dxdy\Bigg|\leq (q+q')\|f\|_{q}\|g\|_{q'},
\end{equation}
where the constant $q+q'$ in (\ref{HLP}) is the best possible, see \cite{HLP}. (\ref{HLP}) has the following equivalent form
\begin{equation}\label{HLP-1}
\left[\int_{\mathbb{R}_{+}}
\left|\int_{\mathbb{R}_{+}}\frac{f(y)}{\max\{x,y\}}dy\right|^q
dx\right]^{\frac{1}{q}}\leq (q+q')\|f\|_{q}.
\end{equation}
(\ref{HLP-1}) is also called {\bf HLP} inequality and the constant $q+q'$ in (\ref{HLP-1}) is still the best possible.  

We can restate the {\bf HLP} inequality (\ref{HLP-1}) in the language of operator theory. We define {\bf HLP} integral operator $H$ induced by {\bf HLP} kernel $\frac{1}{\max\{x,y\}}$ as 
\begin{equation*}Hf(y):=\int_{\mathbb{R}_{+}} \frac{f(x)}{\max\{x,y\}}dx,\: f \in \mathcal{M}(\mathbb{R}_{+}),\: y\in \mathbb{R}_{+}.
\end{equation*}
Then we have
\begin{theorem}Let $q>1$. Let $H$ be as above. Then $H$ is bounded on $L^q(\mathbb{R}_+)$ and the norm $\|H\|_{L^q\rightarrow L^q}$ of $H$ is $q+q'$. Here
\begin{equation*}
\|H\|_{L^q\rightarrow L^q}=\sup_{f\in L^q(\mathbb{R}_{+})}
\frac{\|Hf\|_{q}}{\|f\|_{q}}.
\end{equation*}
\end{theorem}
There is a discrete version for the above result.  For a sequence $a=\{a_m\}_{m=1}^{\infty}$, we define {\bf HLP} operator $\mathcal{H}$ as
$${\mathcal{H}}a(n):=\sum_{m=1}^{\infty}\frac{a_m}{\max\{m,n\}}, \,n\in \mathbb{N}.$$
We use $l^q$ to denote the space of sequences of real numbers, i.e.,
\begin{equation*}l^{q}:=\{a=\{a_n\}_{n=1}^{\infty}: \|a\|_{q}=\Big(\sum_{n=1}^{\infty} |a_{n}|^q\Big)^{\frac{1}{q}}<\infty \}.\end{equation*}
Then we have
\begin{theorem}Let $q>1$. Let $\mathcal{H}$ be as above. Then $\mathcal{H}$ is bounded on $l^q$ and the norm $\|\mathcal{H}\|_{l^p \rightarrow l^p}$ of $\mathcal{H}$ is $q+q'$. Here
\begin{equation*}
\|\mathcal{H}\|_{l^p \rightarrow l^p}=\sup_{a\in l^q}
\frac{\|\mathcal{H}a\|_{q}}{\|a\|_{q}}.
\end{equation*}
\end{theorem}

Hardy-Littlewood-P\'olya operator is related to some important topics in analysis and there have been many results about this operator and its analogous and generalizations. The classical results of this operator can be found in the well-known monograph \cite{HLP}.  In the past three decades,  the so-called Hilbert-type operators, including {\bf HLP}-type operators, have been extensively studied by Yang and his coauthors,  see the survey \cite{YR} and Yang's book \cite{Y3}. For more recent results see for example \cite{WHY}, \cite{YZ}, \cite{J-5}. Very recently, in the work \cite{B}, Brevig established some norm estimates for certain {\bf HLP}-type operators in terms of Riemann zeta function.  Some further results can be found in \cite{B-1}. 

In 2013, Fu et al. initiated the study of $p$-adic {\bf HLP} integral operator in \cite{FWL}, where they introduced and studied a generalized $p$-adic {\bf HLP}-type integral operator and obtained some $p$-adic {\bf HLP}-type inequalities with the best constant factor.  Later, in \cite{J-1} and \cite{J-2},  we studied a class of $p$-adic integral operators induced by a symmetric homogeneous kernel of degree $-1$ and established sharp norm estimates for these operators.  As a byproduct, we found some new $p$-adic {\bf HLP}-type inequalities with the best constant factor and their equivalent forms.  More generalized {\bf HLP}-type operators were studied in \cite{J-3}, \cite{WF} and other related results can be found in \cite{BB}, \cite{DD1,DD2, DD3} and \cite{J-4}.

We note that previous studies have primarily focused on the $p$-adic {\bf HLP}-type operators acting from $p$-adic weighted Lebesgue space $L^{q}_{w_1}(\mathbb{Q}_p^{*})$ to $L^{r}_{w_2}(\mathbb{Q}_p^{*})$ for $1<q=r<\infty$(see below for the definitions of $L_{w_1}^q(\mathbb{Q}_p^{*})$ and $L^{r}_{w_2}(\mathbb{Q}_p^{*})$).  In this paper, by introducing some parameters, we define and study certain generalized $p$-adic {\bf HLP}-type integral operators acting from $L^{q}_{w_1}(\mathbb{Q}_p^{*})$ to $L^{r}_{w_2}(\mathbb{Q}_p^{*})$ for  $1\leq q, r\leq \infty$. We completely characterize the boundedness of these operators in terms of the parameters. We also obtain sharp norm estimates of the operators for some special cases. Our method is based on the ideas of generalized Schur’s tests, see \cite{O}, \cite{Tao}, and \cite{Zh}. The main results of this paper are not only a complement to some previous results but also an extension of existing ones in the literature. These results also enrich our understanding of the theory of $p$-adic {\bf HLP}-type operators. It should be pointed out that this paper is mainly inspired by the work \cite{BS}, although we will use different methods to prove our theorems. 

To state our theorems, we first recall some basic definitions and notations on $p$-adic analysis. We refer the reader to \cite{VVZ} or \cite{T} for a more detailed introduction to the $p$-adic analysis.  

Throughout this paper, we denote by $p$ a prime number. Let $\mathbb{Q}_p$ be the field of $p$-adic numbers, which is defined as the completion 
of the field of rational numbers $\mathbb{Q} $ with respect to the non-Archimedean $p$-adic norm $|\cdot |_p$. 
The $p$-adic norm is defined as follows: {(\bf 1)} $|0|_p = 0$; {(\bf 2)} If any non-zero rational number $x$ is represented as 
$x = p^{\gamma} \frac{m}{n}$, where $\gamma \in \mathbb{Z}$, $m \in \mathbb{Z}$, $n \in \mathbb{Z}$, and $m$ and $n$ are not divisible by $p$, then 
$|x|_p = p^{-\gamma}$.  Any non-zero $p$-adic number $x\in \mathbb{Q}_p$ can be 
uniquely represented in the following canonical form 
\begin{equation}\label{qp-1} 
x=p^{\gamma} \sum_{j=0}^{\infty} a_{j}p^j, \quad \gamma=\gamma(x) \in \mathbb{Z}, 
\end{equation} 
where $a_j$ are integers with $0\leq a_j \leq p-1, a_0 \neq 0$. Since $|a_j p^j|_p=p^{-j}$, the series (\ref{qp-1}) converges in the $p$-adic norm. 

Also, it is not hard to show that the norm satisfies the following properties: 
$$|xy|_{p}= |x|_{p} |y|_{p},\quad |x+y|_{p} \leq \max\{|x|_p, |y|_p\}.$$ 
It follows that,  if $|x|_{p}\neq |y|_{p}$, then $|x+y|_{p} =\max\{|x|_p, |y|_p\}$. 

We set $\mathbb{Q}_p^{*}:=\mathbb{Q}_p \backslash \{0\}$ and denote by 
$$B_{\gamma}(x)=\{y\in \mathbb{Q}_{p}: |y-x|_p \leq p^{\gamma}\},$$ 
the ball with center at $x\in \mathbb{Q}_{p}$ and radius $p^{\gamma}$, and 
$$S_{\gamma}(x)=\{y\in \mathbb{Q}_{p}: |y-x|_p =p^{\gamma}\}=B_{\gamma}(x)\backslash B_{\gamma-1}(x).$$ 
For simplicity, we use $B_{\gamma} $ and $ S_{\gamma}$ to denote $B_{\gamma} (0)$ and $ S_{\gamma}(0)$, respectively. 

Since $\mathbb{Q}_p$ is a locally compact Hausdorff space, then, by the standard theory of measure,  there exists a Haar measure $dx$ on $\mathbb{Q}_p$, which is 
unique up to positive constant multiple and is translation invariant. We normalize the measure $dx$ by the equality
\begin{equation*}\int_{B_{0}}dx=|B_{0}|_H=1,\end{equation*} 
where $|E|_H$ denotes the Haar measure of a measurable subset $E$ of $\mathbb{Q}_p$. A simple calculation yields that 
$$\int_{B_{\gamma}}dx=|B_{\gamma}|_H = p^{\gamma}, \int_{S_{\gamma}}dx=|S_{\gamma}|_H = p^{\gamma}(1-p^{-1}).$$ 

We use $\mathcal{M}(\mathbb{Q}_p^{*})$ to denote the class of all real-valued measurable functions on $\mathbb{Q}_p^{*}$. Let $1\leq q<\infty$, $w(x)$ be a non-negative measurable function on $\mathbb{Q}_p^{*}$, we define the weighted Lebesgue space 
$L_w^q:= L_w^q(\mathbb{Q}_p^{*})$ on $\mathbb{Q}_p^{*}$ as 
\begin{displaymath} 
L_w^q(\mathbb{Q}_p^{*})=\{f\in \mathcal{M}(\mathbb{Q}_p^{*}): 
||f||_{q,w}=\left[\int_{\mathbb{Q}_p^{*}} 
|f(x)|^q w(x)dx \right]^{\frac{1}{q}}<\infty \}. 
\end{displaymath} 
When $w(x)\equiv1$, we will write $L^q$ and $\|f\|_{q}$ instead of $L_w^q$ and $\|f\|_{q,w}$, respectively. 
When $w(x)=|x|_{p}^{\theta}$, $\theta\in \mathbb{R}$, we will write $L_{\theta}^q$ instead of $L_w^q$ and use $\|f\|_{q, \theta}$ to denote $\|f\|_{q, w}$. The class $L^{\infty}:=L^{\infty}(\mathbb{Q}_p^{*})$ on $\mathbb{Q}_p^{*}$ is defined as 
\begin{displaymath} 
L^{\infty}(\mathbb{Q}_p^{*})=\{f\in \mathcal{M}(\mathbb{Q}_p^{*}): \|f\|_{\infty}:=\mathop{\text {ess sup}}\limits_{x \in \mathbb{Q}_p^{*}}|f(x)|<\infty\}.
\end{displaymath}

Let $\lambda, \mu, \nu$ be real numbers. For $f\in \mathcal{M}(\mathbb{Q}_p^{*})$, we define {\bf HLP}-type operator $H_{\lambda, \mu, \nu}$ as 
\begin{equation}\label{ope}H_{\lambda, \mu, \nu}f(y):=\int_{\mathbb{Q}_{p}^{*}}\frac{|x|_p^{\mu}|y|_p^{\nu}}{[\max\{|x|_p, |y|_{p}\}]^{\lambda}}f(x)dx,\,\,\, y\in \mathbb{Q}_{p}^{*}.\end{equation}

On the boundeness of $H_{\lambda, \mu, \nu}$ for $1\leq q\leq r \leq \infty$, we shall prove the following theorems. 

\begin{theorem}\label{th-1}
Let $1\leq q\leq r<\infty$. Let $\lambda, \mu, \nu, \alpha, \beta$ be real numbers and $H_{\lambda, \mu, \nu}$ be as above. Then $H_{\lambda, \mu, \nu}$ is bounded from
$L_{\alpha}^q$ to $L_{\beta}^r$ if and only if 
\begin{equation}\label{th-1-1}
\begin{cases}
\lambda=\mu+\nu+1+\frac{\beta+1}{r}-\frac{\alpha+1}{q}, \\
-r\nu<\beta+1<r(\lambda-\nu), \nonumber 
\end{cases}
\end{equation}
or equivalently, 
\begin{equation}\label{th-1-2}
\begin{cases}
\lambda=\mu+\nu+1+\frac{\beta+1}{r}-\frac{\alpha+1}{q}, \\
q(\mu+1-\lambda)<\alpha+1<q(\mu+1). \nonumber 
\end{cases}
\end{equation}
\end{theorem}

\begin{theorem}\label{th-2}
Let $\lambda, \mu, \nu, \alpha$ be real numbers and $H_{\lambda, \mu, \nu}$ be as above. Then $H_{\lambda, \mu, \nu}$ is bounded from
$L_{\alpha}^1$ to $L^{\infty}$ if and only if 
\begin{equation}\label{th-2-1}
\begin{cases}
\lambda=\mu+\nu-\alpha, \\
\mu-\lambda\leq\alpha\leq \mu, \nonumber 
\end{cases}
\end{equation}
or equivalently, 
\begin{equation}\label{th-2-2}
\begin{cases}
\lambda=\mu+\nu-\alpha, \\
0\leq\nu\leq \lambda. \nonumber 
\end{cases}
\end{equation}
\end{theorem}

\begin{theorem}\label{th-3}
Let $1<q<\infty$. Let $\lambda, \mu, \nu, \alpha$ be real numbers and $H_{\lambda, \mu, \nu}$ be as above. Then $H_{\lambda, \mu, \nu}$ is bounded from
$L_{\alpha}^q$ to $L^{\infty}$ if and only if 
\begin{equation}\label{th-3-1}
\begin{cases}
\lambda=\mu+\nu+1-\frac{\alpha+1}{q}, \\
q(\mu+1-\lambda)<\alpha+1<q(\mu+1), \nonumber 
\end{cases}
\end{equation}
or equivalently, 
\begin{equation}\label{th-3-2}
\begin{cases}
\lambda=\mu+\nu+1-\frac{\alpha+1}{q}, \\
0<\nu<\lambda. \nonumber 
\end{cases}
\end{equation}
\end{theorem}

Next, for some special cases, we shall give sharp estimates for the norm of $H_{\lambda, \mu, \nu}$. 
\begin{theorem}\label{th-4}
Let $1\leq q<\infty$. Let $\lambda, \mu, \nu, \alpha, \beta$ be real numbers and $H_{\lambda, \mu, \nu}$ be as above. 
Then $H_{\lambda, \mu, \nu}$ is bounded from $L_{\alpha}^{q}$ to $L_{\beta}^{q}$ if and only if 
\begin{equation} 
\begin{cases}
\lambda=\mu+\nu+1+\frac{\beta-\alpha}{q}, \\
-q\nu<\beta+1<q(\lambda-\nu), \nonumber 
\end{cases}
\end{equation}
or equivalently, 
\begin{equation} 
\begin{cases}
\lambda=\mu+\nu+1+\frac{\beta-\alpha}{q}, \\
q(\mu+1-\lambda)<\alpha+1<q(\mu+1). \nonumber 
\end{cases}
\end{equation}
Moreover, when $H_{\lambda, \mu, \nu}$ is bounded from $L_{\alpha}^{q}$ to $L_{\beta}^{q}$, the norm of $H_{\lambda, \mu, \nu}$ is given by
 $$\|H_{\lambda, \mu, \nu}\|_{L_{\alpha}^{q}\rightarrow L_{\beta}^q}=(1-p^{-1})\Big[1+\frac{1}{p^{\mu+1-\frac{1}{q}(\alpha+1)}-1}+\frac{1}{p^{\nu+\frac{1}{q}(\beta+1)}-1}\Big].$$ 
Here, $$\|H_{\lambda, \mu, \nu}\|_{L_{\alpha}^{q}\rightarrow L_{\beta}^q}=\sup_{f\in L_{\alpha}^{q}}\frac{\|H_{\lambda,\mu,\nu}f\|_{q, \beta}}{\|f\|_{q, \alpha}}.$$
\end{theorem}
\begin{remark}
 Theorem \ref{th-4} extend some related results in \cite{FWL}, where the special case for $\lambda=1, \mu=\nu=0, \alpha=\beta$ in Theorem \ref{th-4} has been established by Fu et. al. 
\end{remark}
\begin{theorem}\label{th-5}
Let $\lambda, \mu, \nu$ be real numbers and $H_{\lambda, \mu, \nu}$ be as above. Then $H_{\lambda, \mu, \nu}$ is bounded on $L^{\infty}$ if and only if 
\begin{equation} 
\begin{cases}
\lambda=\mu+\nu+1, \\
0<\nu+1<\lambda, \nonumber 
\end{cases}
\end{equation}
or equivalently, 
\begin{equation} 
\begin{cases}
\lambda=\mu+\nu+1, \\
0<\mu<\lambda. \nonumber 
\end{cases}
\end{equation}
Moreover, when $H_{\lambda, \mu, \nu}$ is bounded on $L^{\infty}$, the norm of $H_{\lambda, \mu, \nu}$ is given by
 $$\|H_{\lambda, \mu, \nu}\|_{L^{\infty}\rightarrow L^{\infty}}=(1-p^{-1})\Big[1+\frac{1}{p^{\nu+1}-1}+\frac{1}{p^{\lambda-\nu-1}-1}\Big].$$ 
Here, $$\|H_{\lambda, \mu, \nu}\|_{L^{\infty}\rightarrow L^{\infty}}=\sup_{f\in L^{\infty}}\frac{\|H_{\lambda,\mu,\nu}f\|_{\infty}}{\|f\|_{\infty}}.$$
\end{theorem}

The paper is organized as follows.  We will establish some lemmas in the next section. We shall give the proof of Theorem \ref{th-1} in Section 3. In Section 4, we prove Theorem\ref{th-2} and \ref{th-3}.  We will prove Theorem \ref{th-4} and \ref{th-5} in Section 5 and 6, respectively. In Section 7, we present the final remarks, which deal with the boundedness of $H_{\lambda, \mu, \nu}$ for the remaining cases $1\leq r<q\leq \infty$.

\section{{\bf Some lemmas}}
To prove our main results of this paper, in this section, we recall some known lemmas and establish some new ones. We first recall the following two lemmas, see \cite[Problem 5.5]{Tao} for their general forms. 
\begin{lemma}\label{ll-1}
Let $1\leq q <\infty$. Let $K(x,y)$ be a non-negative measurable function on $\mathbb{Q}_p^{*}\times \mathbb{Q}_p^{*}$. For $f\in \mathcal{M}(\mathbb{Q}_p^{*})$, let $T$ be the integral operator with kernel $K$ defined as
$$Tf(y)=\int_{\mathbb{Q}_p^{*}}K(x,y)f(x)dx,\,\,\, y\in \mathbb{Q}_p^{*}.$$
Then $T$ is bounded from $L^{q}$ to $L^{1}$ if and only if 
$$\int_{\mathbb{Q}_p^{*}}K(x,y)dy\in L^{q'}.$$
Moreover, when $T$ is bounded from $L^{q}$ to $L^{1}$, the norm of $T$ is given by
$$\|T\|_{L^{q}\rightarrow L^1}=\|\int_{\mathbb{Q}_p^{*}}K(x,y)dy\|_{q'}.$$
\end{lemma}
\begin{remark}In particular, when $q=1$ so that $q'=\infty$, $T$ is bounded on $L^{1}$ if and only if 
$$\sup_{x\in \mathbb{Q}_p^{*}}\int_{\mathbb{Q}_p^{*}}K(x,y)dy<\infty.$$ 
Moreover, when $T$ is bounded on $L^{1}$, the norm of $T$ is given by
$$\|T\|_{L^{1}\rightarrow L^1}=\sup_{x\in \mathbb{Q}_p^{*}}\int_{\mathbb{Q}_p^{*}}K(x,y)dy.$$
\end{remark}
The following is a dual version of Lemma \ref{ll-1}. 

\begin{lemma}\label{ll-2}
Let $1\leq q \leq \infty$. Let $K(x,y)$ be a non-negative measurable function on $\mathbb{Q}_p^{*}\times \mathbb{Q}_p^{*}$. For $f\in \mathcal{M}(\mathbb{Q}_p^{*})$, let $T$ be the integral operator with kernel $K$ defined as
$$Tf(y)=\int_{\mathbb{Q}_p^{*}}K(x,y)f(x)dx,\,\,\, y\in \mathbb{Q}_p^{*}.$$
Then $T$ is bounded from $L^{\infty}$ to $L^{q}$ if and only if 
$$\int_{\mathbb{Q}_p^{*}}K(x,y)dx\in L^{q}.$$
Moreover, when  $T$ is bounded from $L^{\infty}$ to $L^{q}$, the norm of $T$ is given by
$$\|T\|_{L^{\infty}\rightarrow L^{q}}=\|\int_{\mathbb{Q}_p^{*}}K(x,y)dx\|_{q}.$$
\end{lemma}
\begin{remark}In particular, when $q=\infty$, $T$ is bounded on $L^{\infty}(\mathbb{Q}_p^{*})$ if and only if 
$$\sup_{y\in \mathbb{Q}_p^{*}}\int_{\mathbb{Q}_p^{*}}K(x,y)dx<\infty.$$
Moreover, when $T$ is bounded on $L^{\infty}$, the norm of $T$ is given by
$$\|T\|_{L^{\infty}\rightarrow L^{\infty}}=\sup_{y\in \mathbb{Q}_p^{*}}\int_{\mathbb{Q}_p^{*}}K(x,y)dx.$$\end{remark}

We will need the following lemmas in our later arguments.
\begin{lemma}\label{ll-3}Let $a, \lambda$ be two real numbers. Let $y\in \mathbb{Q}_{p}^{*}$. Then the integral 
$$I(y)=\int_{\mathbb{Q}_{p}^{*}}\frac{|x|_p^{a}}{[\max\{|x|_p, |y|_{p}\}]^{\lambda}}dx,$$
converges if and only if $a>-1$ and $\lambda-a-1>0$. When the integral converges, we have
$$I(y)=(1-p^{-1})\Big[1+\frac{1}{p^{a+1}-1}+\frac{1}{p^{\lambda-a-1}-1}\Big]|y|_{p}^{a+1-\lambda}.$$ 
\end{lemma}

\begin{proof}
For $y\in \mathbb{Q}_{p}^{*}$, let $|y|_p=p^{-\gamma_y}$, we have
\begin{eqnarray}\label{m-e}
I(y)&=&|y|_{p}^{-\lambda}\int_{|x|_p\leq |y|_p}|x|_p^{a}dx+\int_{|x|_p>|y|_p}|x|_p^{a-\lambda}dx\nonumber \\
&=& |y|_{p}^{-\lambda}(1-p^{-1})\sum_{-\infty<\gamma\leq -{\gamma}_y}p^{\gamma(a+1)}+(1-p^{-1})\sum_{-\gamma_y<\gamma<\infty}p^{\gamma(a-\lambda+1)} \nonumber \\
&=& |y|_{p}^{-\lambda}(1-p^{-1})\sum_{\gamma={\gamma}_y}^{\infty}p^{-\gamma(a+1)}+(1-p^{-1})\sum_{\gamma=-\gamma_y+1}^{\infty}p^{\gamma(a-\lambda+1)}.
\end{eqnarray}
Then we conclude that $I(y)<\infty$ if and only if $a+1>0$ and $\lambda-a-1>0$. Moreover, when $a+1>0$ and $\lambda-a-1>0$, it follows from (\ref{m-e}) that 
\begin{eqnarray} 
I(y)&=&(1-p^{-1})\Big[|y|_{p}^{-\lambda}\frac{p^{-{\gamma_y}(a+1)}}{1-p^{-(a+1)}}+\frac{p^{(-\gamma_y+1)(a+1-\lambda)}}{1-p^{a+1-\lambda}}\Big]\nonumber\\
&=& (1-p^{-1})\Big[\frac{|y|_{p}^{a+1-\lambda}}{1-p^{-(a+1)}}+\frac{|y|_{p}^{a+1-\lambda}p^{a+1-\lambda}}{1-p^{a+1-\lambda}}\Big]\nonumber\\
&=& (1-p^{-1})\Big[1+\frac{1}{p^{a+1}-1}+\frac{1}{p^{\lambda-a-1}-1}\Big]|y|_{p}^{a+1-\lambda}.\nonumber 
\end{eqnarray} 
This proves the lemma. 
\end{proof}

\begin{remark}
When $I(y)$ converges, by a change of variables, it is easy to see that $I(y)$ has the following form
\begin{eqnarray}I(y)&=&|y|_{p}^{a+1-\lambda}\int_{\mathbb{Q}_{p}^{*}}\frac{|t|_p^{a}}{[\max\{1, |t|_{p}\}]^{\lambda}}dt\nonumber \\
&=&(1-p^{-1})\Big[1+\frac{1}{p^{a+1}-1}+\frac{1}{p^{\lambda-a-1}-1}\Big]|y|_{p}^{a+1-\lambda},\, y\in \mathbb{Q}_{p}^{*}.\nonumber\end{eqnarray}
\end{remark}
 
\begin{lemma}\label{ll-l}
Let $\lambda>0$. Let $a, b$ be real numbers. Then there is a constant $C>0$ such that
 $$\sup_{x\in \mathbb{Q}_{p}^{*}}\frac{|x|_p^{a}}{[\max\{|x|_p, |y|_{p}\}]^{\lambda}}\leq C|y|_p^b,$$
for all $y\in \mathbb{Q}_{p}^{*}$, if $\lambda=a-b$ and $a\geq 0, b\leq 0.$
\end{lemma}
\begin{proof}
It suffices to prove that, if $\lambda=a-b$ and $a\geq 0, b\leq 0$, then 
$$\mathcal{S}=\sup_{x, y\in \mathbb{Q}_{p}^{*}}\frac{|x|_p^{a}|y|_p^{-b}}{[\max\{|x|_p, |y|_{p}\}]^{\lambda}}<\infty.$$
Actually, when $\lambda=a-b$ and $a\geq 0, b\leq 0$, we have
\begin{eqnarray}
\mathcal{S}_1:=\sup_{0<|x_p|\leq |y|_p<\infty}\frac{|x|_p^{a}|y|_p^{-b}}{[\max\{|x|_p, |y|_{p}\}]^{\lambda}}=\sup_{0<|x_p\leq |y|_p<\infty}\Big(\frac{|x|_p}{|y|_p}\Big)^{a}<\infty, \nonumber 
\end{eqnarray}
and
\begin{eqnarray}
\mathcal{S}_2:=\sup_{0<|y|_p<|x|_p<\infty}\frac{|x|_p^{a}|y|_p^{-b}}{[\max\{|x|_p, |y|_{p}\}]^{\lambda}}=\sup_{0<|y|_p<|x|_p<\infty}\Big(\frac{|y|_p}{|x|_p}\Big)^{-b}<\infty. \nonumber 
\end{eqnarray}
The lemma then follows from the fact that $\mathcal{S}=\max\{\mathcal{S}_1, \mathcal{S}_2\}$. 
\end{proof}
The following lemma provides the necessary conditions for the boundedness of $H_{\lambda, \mu, \nu}$. 
\begin{lemma}\label{ll-4}
Let $1\leq q, r<\infty$. Let $\lambda, \mu, \nu, \alpha, \beta$ be real numbers and $H_{\lambda, \mu, \nu}$ be as in (\ref{ope}).  If $H_{\lambda, \mu, \nu}$ is bounded from 
$L_{\alpha}^q$ to $L_{\beta}^r$ , then we have 
\begin{equation} 
\begin{cases}
\lambda=\mu+\nu+1+\frac{\beta+1}{r}-\frac{\alpha+1}{q}, \\
-r\nu<\beta+1<r(\lambda-\nu), \nonumber 
\end{cases}
\end{equation}
or equivalently, 
\begin{equation} 
\begin{cases}
\lambda=\mu+\nu+1+\frac{\beta+1}{r}-\frac{\alpha+1}{q}, \\
q(\mu+1-\lambda)<\alpha+1<q(\mu+1). \nonumber 
\end{cases}
\end{equation} 
\end{lemma}
 
\begin{proof}
We let 
$$\tau:=\mu+\nu+1+\frac{\beta+1}{r}-\frac{\alpha+1}{q}-\lambda.$$  
We suppose that $H_{\lambda, \mu, \nu}$ is bounded from $L_{\alpha}^{q}$ to $L_{\beta}^{r}$. 
If $\tau<0$, for $\varepsilon>0$, we let 
\begin{equation}\label{f-1}f_{\varepsilon}^{[1]}(x)= 
\begin{cases}
|x|_{p}^{-\frac{\alpha+1}{q}+\frac{\varepsilon}{q}},\,\, {\text{when}}\, 0<|x|_p\leq 1, \\
0,\quad\quad\quad\quad\,\,\,{\text{when}}\, |x|_p>1.
\end{cases}
\end{equation}
Then we have $$\|f_{\varepsilon}^{[1]}\|_{q, \alpha}^q=\frac{1-p^{-1}}{1-p^{-\varepsilon}}.$$
From the boundedness of $H_{\lambda, \mu, \nu}$, we know that, for any $y\in \mathbb{Q}_p^{*}$,  
$$H_{\lambda, \mu, \nu}f_{\varepsilon}^{[1]}(y)=\int_{|x|_p\leq 1}\frac{|x|_p^{\mu}|y|_p^{\nu}}{[\max\{|x|_p, |y|_{p}\}]^{\lambda}}|x|_{p}^{-\frac{1+\alpha}{q}+\frac{\varepsilon}{q}}dx<\infty.$$
Consequently, when $|y|_p\geq 1$, we have 
\begin{eqnarray}H_{\lambda, \mu, \nu}f^{[1]}(y)&=&|y|_p^{\nu-\lambda}\int_{|x|_p\leq 1}|x|_p^{\mu-\frac{\alpha+1}{q}+\frac{\varepsilon}{q}}dx\nonumber \\
&=& |y|_p^{\nu-\lambda}\sum_{\gamma=0}^{\infty}p^{-\gamma(\mu+1-\frac{\alpha+1}{q}+\frac{\varepsilon}{q})}<\infty.\nonumber 
\end{eqnarray}
This implies that $\mu+1-\frac{\alpha+1}{q}+\frac{\varepsilon}{q}>0$ for any $\varepsilon>0$ so that
$$\mu+1-\frac{\alpha+1}{q}\geq 0.$$  
Meanwhile, we have 
\begin{eqnarray}\label{h-1}\|H_{\lambda, \mu, \nu}f_{\varepsilon}^{[1]}\|_{r,\beta}^r&=&\int_{\mathbb{Q}_p^{*}}
|y|_p^{\beta}\left [\int_{0<|x|_p \leq 1}\frac{|x|_p^{\mu}|y|_p^{\nu}}{[\max\{|x|_p, |y|_{p}\}]^{\lambda}}|x|_{p}^{-\frac{\alpha+1}{q}+\frac{\varepsilon}{q}}dx\right ]^rdy \nonumber\\ &=&
\int_{\mathbb{Q}_p^{*}}
|y|_p^{-1+\frac{r}{q}\varepsilon+r\tau}\left[\int_{0<|t|_p \leq |y|_p^{-1}}\frac{|t|_p^{\mu-\frac{\alpha+1}{q}+\frac{\varepsilon}{q}}}{[\max\{1, |t|_{p}\}]^{\lambda}}
dt\right]^r dy \nonumber \\
&\geq &\int_{|y|_p \leq 1}
|y|_p^{-1+\frac{r}{q}\varepsilon+r\tau}\left[\int_{0<|t|_p \leq 1}\frac{|t|_p^{\mu-\frac{\alpha+1}{q}+\frac{\varepsilon}{q}}}{[\max\{1, |t|_{p}\}]^{\lambda}}
dt\right]^r dy\nonumber \\
&=&(1-p^{-1})\sum_{\gamma=0}^{\infty}p^{\gamma(-\frac{r}{q}\varepsilon-r\tau)}\left[\int_{0<|t|_p \leq 1}\frac{|t|_p^{\mu-\frac{\alpha+1}{q}+\frac{\varepsilon}{q}}}{[\max\{1, |t|_{p}\}]^{\lambda}}
dt\right]^r. \nonumber 
\end{eqnarray}
Hence when $\varepsilon<-q\tau$, we obtain that
$$\sum_{\gamma=0}^{\infty}p^{\gamma(-\frac{r}{q}\varepsilon-r\tau)}=\infty,$$
so that $\|H_{\lambda, \mu, \nu}f_{\varepsilon}^{[1]}\|_{r,\beta}^r=\infty$, which contradicts the boundedness of $H_{\lambda, \mu, \nu}$. This implies that $\tau\geq 0$. 

Furthermore, if $\tau>0$, for $\varepsilon>0$, we take 
\begin{equation}\label{f-2}{f}^{[2]}_\varepsilon(x)=\begin{cases}
0, \qquad\quad\;\;\;\;\,  \text{when} \; 0<|x|_p <1, \\
|x|_{p}^{-\frac{\alpha+1}{q}-\frac{\varepsilon}{q}}, \; \text{when} \; |x|_p \geq 1.\\
\end{cases}
\end{equation}
Then 
$$\|f_\varepsilon^{[2]}\|_{q,\alpha}^{q}=\frac{1-p^{-1}}{1-p^{-\varepsilon}},$$
It follows from the boundedness of $H_{\lambda, \mu, \nu}$ that, for any $y\in \mathbb{Q}_p^{*}$,  
$$H_{\lambda, \mu, \nu}f_{\varepsilon}^{[2]}(y)=\int_{|x|_p\geq 1}\frac{|x|_p^{\mu}|y|_p^{\nu}}{[\max\{|x|_p, |y|_{p}\}]^{\lambda}}|x|_{p}^{-\frac{\alpha+1}{q}-\frac{\varepsilon}{q}}dx<\infty.$$
Consequently, when $|y|_p\leq 1$, we have 
\begin{eqnarray}H_{\lambda, \mu, \nu}f^{[2]}(y)&=&|y|_p^{\nu}\int_{|x|_p\geq 1}|x|_p^{\mu-\lambda-\frac{\alpha+1}{q}-\frac{\varepsilon}{q}}dx
\nonumber \\
&=& |y|_p^{\nu} \sum_{\gamma=0}^{\infty} p^{\gamma(\mu+1-\lambda-\frac{\alpha+1}{q}-\frac{\varepsilon}{q})}<\infty.\nonumber \end{eqnarray}
This implies that $\mu+1-\lambda-\frac{\alpha+1}{q}-\frac{\varepsilon}{q}<0$ for any $\varepsilon>0$ so that
$$\mu+1-\lambda-\frac{\alpha+1}{q}\leq 0,$$  
and 
\begin{eqnarray}\label{h-2}\|H_{\lambda, \mu, \nu}f_{\varepsilon}^{[2]}\|_{r,\beta}^r&=&\int_{\mathbb{Q}_p^{*}}
|y|_p^{\beta}\left [\int_{|x|_p \geq 1}\frac{|x|_p^{\mu}|y|_p^{\nu}}{[\max\{|x|_p, |y|_{p}\}]^{\lambda}}|x|_{p}^{-\frac{\alpha+1}{q}-\frac{\varepsilon}{q}}dx\right ]^rdy \nonumber \\ &=&
\int_{\mathbb{Q}_p^{*}}
|y|_p^{-1-\frac{r}{q}\varepsilon+r\tau}\left[\int_{|t|_p \geq |y|_p^{-1}}\frac{|t|_p^{\mu-\frac{\alpha+1}{q}-\frac{\varepsilon}{q}}}{[\max\{1, |t|_{p}\}]^{\lambda}}
dt\right]^r dy \nonumber \\
&\geq &\int_{|y|_p\geq 1}
|y|_p^{-1-\frac{r}{q}\varepsilon+r\tau}\left[\int_{|t|_p \geq 1}\frac{|t|_p^{\mu-\frac{\alpha+1}{q}-\frac{\varepsilon}{q}}}{[\max\{1, |t|_{p}\}]^{\lambda}}
dt\right]^r dy\nonumber \\
&=&(1-p^{-1})\sum_{\gamma=0}^{\infty}p^{\gamma(-\frac{r}{q}\varepsilon+r\tau)}\left[\int_{|t|_p \geq 1}\frac{|t|_p^{\mu-\frac{\alpha+1}{q}-\frac{\varepsilon}{q}}}{[\max\{1, |t|_{p}\}]^{\lambda}}
dt\right]^r.\nonumber 
\end{eqnarray}
Thus when $\varepsilon<q\tau$, we obtain that $$\sum_{\gamma=0}^{\infty}p^{\gamma(-\frac{r}{q}\varepsilon+r\tau)}=\infty,$$
so that $\|H_{\lambda, \mu, \nu}f_{\varepsilon}^{[1]}\|_{r,\beta}^r=\infty$, which contradicts to the boundedness of $H_{\lambda, \mu, \nu}$. This means that $\tau$ can only take value $0$. 

Collecting all above arguments, we have proved that, if $H_{\lambda, \mu, \nu}$ is bounded from $L_{\alpha}^{q}$ to $L_{\beta}^{r}$, then  
\begin{equation}
\begin{cases}
\lambda=\mu+\nu+1+\frac{\beta+1}{r}-\frac{\alpha+1}{q}, \\
q(\mu+1-\lambda)\leq\alpha+1\leq q(\mu+1).
\end{cases}\nonumber
\end{equation}
Finally, to finish the proof, we have to prove that, if \begin{equation}\label{cc-1}\begin{cases}
\lambda=\mu+\nu+1+\frac{\beta+1}{r}-\frac{\alpha+1}{q}, \\
\alpha+1=q(\mu+1),
\end{cases}
\end{equation}
or \begin{equation}\label{cc-2}
\begin{cases}
\lambda=\mu+\nu+1+\frac{\beta+1}{r}-\frac{\alpha+1}{q}, \\
q(\mu+1-\lambda)=\alpha+1.
\end{cases}
\end{equation}
Then the operator $H_{\lambda, \mu, \nu}$ can not be bounded from $L_{\alpha}^{q}$ to $L_{\beta}^{r}$. 

First, if (\ref{cc-1}) holds, for $\varepsilon>0$, take $f_{\varepsilon}^{[1]}$ as in (\ref{f-1}), then we obtain that
\begin{eqnarray}\|H_{\lambda, \mu, \nu}f_{\varepsilon}^{[1]}\|_{r,\beta}^r&=&\int_{\mathbb{Q}_p^{*}}
|y|_p^{\beta}\left [\int_{0<|x|_p \leq 1}\frac{|x|_p^{\mu}|y|_p^{\nu}}{[\max\{|x|_p, |y|_{p}\}]^{\lambda}}|x|_{p}^{-\frac{\alpha+1}{q}+\frac{\varepsilon}{q}}dx\right ]^rdy\nonumber  \\ 
&=& \int_{\mathbb{Q}_p^{*}}
|y|_p^{-1+\frac{r}{q}\varepsilon}\left[\int_{0<|t|_p \leq |y|_p^{-1}}\frac{|t|_p^{-1+\frac{\varepsilon}{q}}}{[\max\{1, |t|_{p}\}]^{\lambda}}
dt\right]^r dy \nonumber \\
&\geq &\int_{|y|_p \geq 1}
|y|_p^{-1+\frac{r}{q}\varepsilon}\left[\int_{0<|t|_p \leq |y|_p^{-1}}|t|_p^{-1+\frac{\varepsilon}{q}}
dt\right]^r dy\nonumber.
\end{eqnarray}
Meanwhile, for $|y|_p \geq 1$, let $|y|_p=p^{-\gamma_y}$, we have
 $$\int_{0<|t|_p \leq |y|_p^{-1}}|t|_p^{-1+\frac{\varepsilon}{q}}
dt=(1-p^{-1})\sum_{-\infty<\gamma\leq \gamma_y}p^{\gamma\frac{\varepsilon}{q}}=\frac{1-p^{-1}}{1-p^{-\frac{\varepsilon}{q}}}|y|_p^{-\frac{\varepsilon}{q}}.$$
Consequently, we have 
\begin{eqnarray}\|H_{\lambda, \mu, \nu}f_{\varepsilon}^{[1]}\|_{r,\beta}^r \geq \Big[\frac{1-p^{-1}}{1-p^{-\frac{\varepsilon}{q}}}\Big]^r\int_{|y|_p\geq 1} |y|_p^{-1}dy=\infty.\nonumber\end{eqnarray}
This means that $H_{\lambda, \mu, \nu}$ is unbounded from $L_{\alpha}^{q}$ to $L_{\beta}^{r}$ when (\ref{cc-1}) holds.
 
Second, if (\ref{cc-2}) holds, for $\varepsilon>0$, take $f_{\varepsilon}^{[2]}$ as in (\ref{f-2}), then we have
\begin{eqnarray}\label{h-2}\|H_{\lambda, \mu, \nu}f_{\varepsilon}^{[2]}\|_{r,\beta}^r&=&\int_{\mathbb{Q}_p^{*}}
|y|_p^{\beta}\left [\int_{|x|_p \geq 1}\frac{|x|_p^{\mu}|y|_p^{\nu}}{[\max\{|x|_p, |y|_{p}\}]^{\lambda}}|x|_{p}^{-\frac{\alpha+1}{q}-\frac{\varepsilon}{q}}dx\right ]^rdy \nonumber \\ &=&
\int_{\mathbb{Q}_p^{*}}
|y|_p^{-1-\frac{r}{q}\varepsilon+r\tau}\left[\int_{|t|_p \geq |y|_p^{-1}}\frac{|t|_p^{\mu-\frac{\alpha+1}{q}-\frac{\varepsilon}{q}}}{[\max\{1, |t|_{p}\}]^{\lambda}}
dt\right]^r dy \nonumber \\
&\geq &\int_{|y|_p\leq 1}
|y|_p^{-1-\frac{r}{q}\varepsilon}\left[\int_{|t|_p \geq|y|_p^{-1}}|t|_p^{\mu-\lambda-\frac{\alpha+1}{q}-\frac{\varepsilon}{q}}dt\right]^r dy.\nonumber
\end{eqnarray}
Furthermore, for $|y|_p \leq 1$, let $|y|_p=p^{-\gamma_y}$, from (\ref{cc-2}),  we have 
\begin{eqnarray}
\int_{|t|_p \geq|y|_p^{-1}}|t|_p^{\mu-\lambda-\frac{\alpha+1}{q}-\frac{\varepsilon}{q}}dt&=&\int_{|t|_p \geq|y|_p^{-1}}|t|_p^{\mu-\lambda-\frac{\alpha+1}{q}-\frac{\varepsilon}{q}}dt \nonumber \\
&=& \int_{|t|_p \geq|y|_p^{-1}}|t|_p^{-1-\frac{\varepsilon}{q}}dt\nonumber \\
&=&(1-p^{-1})\sum_{\gamma_y<\gamma<\infty}p^{-\gamma\frac{\varepsilon}{q}}=\frac{1-p^{-1}}{1-p^{-\frac{\varepsilon}{q}}}|y|_p^{\frac{\varepsilon}{q}}. \nonumber 
\end{eqnarray}
Consequently, we have 
\begin{eqnarray}\|H_{\lambda, \mu, \nu}f_{\varepsilon}^{[1]}\|_{r,\beta}^r \geq \Big[\frac{1-p^{-1}}{1-p^{-\frac{\varepsilon}{q}}}\Big]^r\int_{|y|_p\leq 1} |y|_p^{-1}dy=\infty,\nonumber\end{eqnarray}
so that  $H_{\lambda, \mu, \nu}$ is unbounded from $L_{\alpha}^{q}$ to $L_{\beta}^{r}$ in this case. Now, we complete the proof of Lemma \ref{ll-4}.
\end{proof}

\section{{\bf Proof of Theorem \ref{th-1}}}
In this section, for the sake of simplicity, we will write
$$k_{\lambda, \mu, \nu}(x,y)=\frac{|x|_p^{\mu}|y|_p^{\nu}}{[\max\{|x|_p, |y|_{p}\}]^{\lambda}}.$$
Note that the necessity of the boundedness of $H_{\lambda, \mu, \nu}$ has been proved by Lemma \ref{ll-4}. We only need to prove that, for $1\leq q\leq r<\infty$, if 
\begin{equation}\label{e-1}\lambda=\mu+\nu+1+\frac{\beta+1}{r}-\frac{\alpha+1}{q},\end{equation}
and
\begin{equation}\label{e-2}-r\nu<\beta+1<r(\lambda-\nu).\end{equation}
Then $H_{\lambda, \mu, \nu}$ is bounded $L_{\alpha}^{q}$ to $L_{\beta}^{r}$. First, we observe from (\ref{e-1}) that (\ref{e-2}) is equivalent to the following inequality 
\begin{equation}\label{e-3}q(\mu+1-\lambda)<\alpha+1<q(\mu+1).\end{equation}
It follows from the left part of (\ref{e-2}) and the right part of (\ref{e-3}) that
$$\lambda=\Big(\nu+\frac{\beta+1}{r}\Big)+\Big(\mu+1-\frac{\alpha+1}{q}\Big)>0.$$

We will divide our proof into the following two cases.
\subsection{\bf Case I.  $1<q\leq r<\infty$} 
As $\beta+1<r(\lambda-\nu)$, i.e., $r(\lambda-\nu)-\beta-1>0$, then we can find a constant $t>1$ such that 
$$r(\lambda-\nu)-\beta-1>\frac{r\lambda}{t}.$$ 
Let $s$ be such that $\frac{1}{s}+\frac{1}{t}=1$. It follows that
\begin{equation}\label{add-1}-\frac{\beta+1}{p}-\nu+\frac{\lambda}{s}>0.\end{equation}  From the left part of (\ref{e-2}), we have
$-\frac{\beta+1}{r}-\nu<0$ so that 
\begin{equation}\label{add-2}-\frac{\beta+1}{r}-\nu-\frac{1}{q'}-\frac{\mu}{s}+\frac{\lambda}{s}<-\frac{1}{q'}-\frac{\mu}{s}+\frac{\lambda}{s}.\end{equation}
Meanwhile, we see from (\ref{add-1}) and $\lambda>0$ that  
\begin{equation}\label{add-3}-\frac{\beta+1}{r}-\nu-\frac{1}{q'}-\frac{\mu}{s}+\frac{\lambda}{s}<-\frac{1}{q'}-\frac{\mu}{s}<-\frac{1}{q'}-\frac{\mu}{s}+\frac{\lambda}{s},\end{equation}
and 
\begin{equation}\label{add-4}-\frac{\beta+1}{r}-\nu-\frac{1}{q'}-\frac{\mu}{s}+\frac{\lambda}{s}<-\frac{\beta+1}{r}-\nu-\frac{1}{q'}-\frac{\mu}{s}+\lambda.\end{equation}
Consequently, in view of (\ref{add-2}), (\ref{add-3}), (\ref{add-4}), we can take a constant $A$ such that 
\begin{equation}\label{e-4} 
-\frac{1}{q'}-\frac{\mu}{s}<A<-\frac{1}{q'}-\frac{\mu}{s}+\frac{\lambda}{s},\end{equation}
and 
\begin{equation}\label{e-5}-\frac{\beta+1}{r}-\nu-\frac{\mu}{s}-\frac{1}{q'}+\frac{\lambda}{s}<A<-\frac{\beta+1}{r}-\nu-\frac{\mu}{s}-\frac{1}{q'}+\lambda.\end{equation}
It is easy to see that (\ref{e-4}) is equivalent to
\begin{equation}\label{equi-1}
\begin{cases}
\frac{\mu}{s}q'+Aq'>-1, \\
\frac{\lambda}{s}q'-\frac{\mu}{s}q'-Aq'-1>0,
\end{cases}
\end{equation}
and (\ref{e-5}) is equivalent to
\begin{equation}\label{equi-2}
\begin{cases}
\frac{\nu}{t}r+B+\beta>-1, \\
\frac{\lambda}{t}r-\frac{\nu}{t}r-B-\beta-1>0.
\end{cases}
\end{equation}
Here 
$$B=r[\frac{\nu}{s}+\frac{\mu}{s}+A+\frac{1}{q'}-\frac{\lambda}{s}].$$

Now, for $0\leq f\in L_{\alpha}^{q}$, we have, for $y\in \mathbb{Q}_p^{*}$, 
\begin{eqnarray}
H_{\lambda, \mu, \nu}f(y)&=&\int_{\mathbb{Q}_p^{*}}k_{\lambda, \mu, \nu}(x,y)f(x)dx\nonumber \\
&=& \int_{\mathbb{Q}_p^{*}}[k_{\lambda, \mu, \nu}(x,y)]^{\frac{1}{s}}x^{A}\cdot[k_{\lambda, \mu, \nu}(x,y)]^{\frac{1}{t}}x^{-A}f(x)dx. \nonumber
\end{eqnarray}
By H\"{o}lder's inequality, we obtain that
\begin{eqnarray}
H_{\lambda, \mu, \nu}f(y)&\leq & \Big\{\int_{\mathbb{Q}_p^{*}}[k_{\lambda, \mu, \nu}(x,y)]^{\frac{q'}{s}}x^{Aq'}dx\Big\}^{\frac{1}{q'}}\Big\{[k_{\lambda, \mu, \nu}(x,y)]^{\frac{q}{t}}x^{-Aq}f^q(x)dx\Big\}^{\frac{1}{q}} \nonumber \\
&=& \Big\{\int_{\mathbb{Q}_p^{*}}\frac{|x|_p^{\frac{\mu}{s}q'+Aq'}|y|_p^{\frac{\nu}{s}q'}}{[\max\{|x|_p, |y|_{p}\}]^{\frac{\lambda}{s}q'}}dx\Big\}^{\frac{1}{q'}}\Big\{\int_{\mathbb{Q}_p^{*}}[k_{\lambda, \mu, \nu}(x,y)]^{\frac{q}{t}}x^{-Aq}f^q(x)dx\Big\}^{\frac{1}{q}}
\nonumber \\
&:=& {\bf I}_1^{^{\frac{1}{q'}}}{\bf I}_2^{^{\frac{1}{q}}}.\nonumber
\end{eqnarray}
By (\ref{equi-1}) and Lemma \ref{ll-3}, we see that 
\begin{eqnarray}
{\bf I}_1&=&\Big[1+\frac{1}{p^{\frac{\mu}{s}q'+Aq'+1}-1}+\frac{1}{p^{\frac{\lambda}{s}q'-\frac{\mu}{s}q'-Aq'-1}-1}\Big]|y|_p^{\frac{\mu}{s}q'+\frac{\nu}{s}q'+Aq'+1-\frac{\lambda}{s}q'}\nonumber \\
&:=&{\bf C}_1 |y|_p^{\frac{\mu}{s}q'+\frac{\nu}{s}q'+Aq'+1-\frac{\lambda}{s}q'}.\nonumber
\end{eqnarray}
Consequently, we obtain that
 \begin{eqnarray}
\|H_{\lambda, \mu, \nu}f\|_{r, \beta}&\leq & \Big[\int_{\mathbb{Q}_p^{*}}{\bf I}_1^{^{\frac{r}{q'}}}{\bf I}_2^{^{\frac{r}{q}}}|y|_{p}^{\beta}dy\Big]^{\frac{1}{r}}\nonumber \\
&=&{\bf C}_1^{\frac{1}{q'}}\Big[\int_{\mathbb{Q}_p^{*}}|y|_p^{r(\frac{\mu}{s}+\frac{\nu}{s}+A+\frac{1}{q'}-\frac{\lambda}{s})}\nonumber \\
&&\quad\quad\quad \times \Big\{\int_{\mathbb{Q}_p^{*}}[k_{\lambda, \mu, \nu}(x,y)]^{\frac{q}{t}}x^{-Aq}f^q(x)dx\Big\}^{\frac{r}{q}}|y|_{p}^{\beta}dy\Big]^{\frac{1}{r}}
\nonumber \\
&=& {\bf C}_1^{\frac{1}{q'}}\Big[\int_{\mathbb{Q}_p^{*}}|y|_p^{B+\beta}\Big\{\int_{\mathbb{Q}_p^{*}}[k_{\lambda, \mu, \nu}(x,y)]^{\frac{q}{t}}x^{-Aq}f^q(x)dx\Big\}^{\frac{r}{q}}dy\Big]^{\frac{1}{r}}. \nonumber
\end{eqnarray}
Then it follows from Minkowski's inequality that 
 \begin{eqnarray}
\|H_{\lambda, \mu, \nu}f\|_{r, \beta}&\leq & {\bf C}_1^{\frac{1}{q'}}\Big[\int_{\mathbb{Q}_p^{*}}|y|_p^{B+\beta}\Big\{\int_{\mathbb{Q}_p^{*}}[k_{\lambda, \mu, \nu}(x,y)]^{\frac{q}{t}}x^{-Aq}f^q(x)dx\Big\}^{\frac{r}{q}}dy\Big]^{\frac{q}{r}\cdot\frac{1}{q}} \nonumber \\
&\leq & {\bf C}_1^{\frac{1}{q'}}\Big[\int_{\mathbb{Q}_p^{*}}\Big\{\int_{\mathbb{Q}_p^{*}}[k_{\lambda, \mu, \nu}(x,y)]^{\frac{r}{t}}|y|_p^{B+\beta}dy\Big\}^{\frac{q}{r}}|x|_p^{-Aq}f^q(x)dx\Big]^{\frac{1}{q}} \nonumber \\
&=& {\bf C}_1^{\frac{1}{q'}}\Big[\int_{\mathbb{Q}_p^{*}}\Big\{\int_{\mathbb{Q}_p^{*}}\frac{|x|_p^{\frac{\mu}{t}r}|y|_p^{\frac{\nu}{t}r+B+\beta}}{[\max\{|x|_p, |y|_{p}\}]^{\frac{\lambda}{t}r}}dy\Big\}^{\frac{q}{r}}|x|_p^{-Aq}f^q(x)dx\Big]^{\frac{1}{q}}\nonumber \\
&:=& {\bf C}_1^{\frac{1}{q'}}\Big[\int_{\mathbb{Q}_p^{*}}{\bf I}_3^{\frac{q}{r}}x^{-Aq}f^q(x)dx\Big]^{\frac{1}{q}}. \nonumber
\end{eqnarray}
Moreover, by using (\ref{equi-2}) and again Lemma \ref{ll-3}, we have   
\begin{eqnarray}
{\bf I}_3&=&\Big[1+\frac{1}{p^{\frac{\nu}{t}r+B+\beta+1}-1}+\frac{1}{p^{\frac{\lambda}{t}r-\frac{\nu}{t}r-B-\beta-1}-1}\Big]|x|_p^{\frac{\nu}{t}r+B+\beta+1-\frac{\lambda}{t}r+\frac{\mu}{t}r}\nonumber \\
&:=& {\bf C}_2 |x|_p^{\frac{\nu}{t}r+B+\beta+1-\frac{\lambda}{t}r+\frac{\mu}{t}r}. \nonumber
\end{eqnarray}
It follows that
 \begin{eqnarray}
\|H_{\lambda, \mu, \nu}f\|_{r, \beta}&\leq &{\bf C}_1^{\frac{1}{q'}}{\bf C}_2^{\frac{1}{r}}\Big[\int_{\mathbb{Q}_p^{*}}|x|_p^{\frac{q}{r}[\frac{\nu}{t}r+B+\beta+1-\frac{\lambda}{t}r+\frac{\mu}{t}r]}|x|_p^{-Aq}f^q(x)dx\Big]^{\frac{1}{q}}. \nonumber
\end{eqnarray}
Meanwhile, we note that
\begin{eqnarray}
\lefteqn{\frac{q}{r}[\frac{\nu}{t}r+B+\beta+1-\frac{\lambda}{t}r+\frac{\mu}{t}r]-Aq=q[\frac{\nu}{t}+\frac{B}{r}+\frac{\beta+1}{r}-\frac{\lambda}{t}+\frac{\mu}{t}-A]}\nonumber \\
&&=q[\frac{\nu}{t}+(\frac{\mu+\nu}{s}+A+\frac{1}{q'}-\frac{\lambda}{s})+\frac{\beta+1}{r}-\frac{\lambda}{t}+\frac{\mu}{t}-A] \nonumber \\
&&=q[\nu+\mu+\frac{1}{q'}-\lambda+\frac{\beta+1}{r}]=q[\frac{1}{q'}-1+\frac{\alpha+1}{q}]=\alpha. \nonumber
\end{eqnarray}
This means that  \begin{eqnarray}
\|H_{\lambda, \mu, \nu}f\|_{r, \beta}&\leq &{\bf C}_1^{\frac{1}{q'}}{\bf C}_2^{\frac{1}{r}}\|f\|_{q, \alpha}, \nonumber
\end{eqnarray}
so that $H_{\lambda, \mu, \nu}$ is bounded from $L_{\alpha}^{q}$ to $L_{\beta}^{r}$ in this case.

\subsection{\bf Case II.  $1=q\leq r<\infty$}
Recall that the following conditions are satisfied. 
\begin{equation}\label{ee-1}\lambda=\mu+\nu+\frac{\beta+1}{r}-\alpha, \nonumber\end{equation}
and
\begin{equation}\label{ee-2}-r\nu<\beta+1<r(\lambda-\nu).\nonumber\end{equation}

As $-r\nu<\beta+1$, then there is a constant $s>1$ such that 
\begin{equation}\label{ine-1}\beta+1>-r\nu+\frac{r\lambda}{s}.\end{equation}

We now let $t>1$ be such that $\frac{1}{s}+\frac{1}{t}=1$. Noting that $\beta+1<r(\lambda-\nu)$ so that 
\begin{equation}\label{ine-2}
\frac{\nu-\lambda}{s}<\frac{\lambda-\nu}{t}-\frac{\beta+1}{r},
\end{equation}
and from (\ref{ine-1}) we have  \begin{equation}\label{ine-3}
-\frac{\nu}{t}-\frac{\beta+1}{r}<\frac{\nu-\lambda}{s}<\frac{\nu}{s}.
\end{equation}
Then we see from (\ref{ine-2}) and (\ref{ine-3}) that we can take a constant $D$ such that
\begin{equation}\label{ine-4}
-\frac{\nu}{t}-\frac{\beta+1}{r}<D<\frac{\lambda-\nu}{t}-\frac{\beta+1}{r},
\end{equation}
and 
\begin{equation}\label{ine-5}
\frac{\nu-\lambda}{s}<D<\frac{\nu}{s}.
\end{equation}
Consequently, it is easy to see that (\ref{ine-4}) is equivalent to 
\begin{equation}\label{equi-3}
\begin{cases}
r\frac{\nu}{t}+rD+\beta+1>0, \\
r\frac{\lambda-\nu}{t}-rD-\beta-1>0, 
\end{cases}
\end{equation}
and  (\ref{ine-5}) is equivalent to  \begin{equation}\label{equi-4}
\begin{cases}
\frac{\lambda-\nu}{s}+D>0, \\
\frac{\nu}{s}-D>0. 
\end{cases}
\end{equation}
Hence we see from (\ref{equi-4}) and Lemma \ref{ll-l} that there is a constant ${\bf C}_3>0$ such that 
\begin{equation}
\sup_{x\in \mathbb{Q}_p^{*}}\frac{|x|_p^{\frac{\lambda-\nu}{s}+D}}{[\max\{|x|_p, |y|_p\}]^{\frac{\lambda}{s}}} \leq {\bf C}_3 |y|_p^{D-\frac{\nu}{s}}, \nonumber 
\end{equation}
for all $y\in \mathbb{Q}_p^{*}$. It follows that, for $0\leq f\in L_{\alpha}^{1}$, $y\in \mathbb{Q}_p^{*}$, 
\begin{eqnarray}
H_{\lambda, \mu, \nu}f(y)&=&\int_{\mathbb{Q}_p^{*}}[k_{\lambda, \mu, \nu}(x,y)]^\frac{1}{s}[k_{\lambda, \mu, \nu}(x,y)]^{\frac{1}{t}}f(x)dx \nonumber \\
&=&\int_{\mathbb{Q}_p^{*}} \frac{|x|_p^{\frac{\lambda-\nu}{s}+D}}{[\max\{|x|_p, |y|_p\}]^{\frac{\lambda}{s}}}|x|_p^{\frac{\mu+\nu-\lambda}{s}-D}|y|_p^{\frac{\nu}{s}}[k_{\lambda, \mu, \nu}(x,y)]^{\frac{1}{t}}f(x)dx \nonumber \\
&\leq & {\bf C}_3 |y|_p^{D}\int_{\mathbb{Q}_p^{*}} |x|_p^{\frac{\mu+\nu-\lambda}{s}-D}[k_{\lambda, \mu, \nu}(x,y)]^{\frac{1}{t}}f(x)dx. \nonumber 
\end{eqnarray}
Consequently, we have
\begin{equation}
\|H_{\lambda, \mu, \nu}f\|_{r, \beta}\leq {\bf C}_3\Big[\int_{\mathbb{Q}_p^{*}}\Big(\int_{\mathbb{Q}_p^{*}} |x|_p^{\frac{\mu+\nu-\lambda}{s}-D}[k_{\lambda, \mu, \nu}(x,y)]^{\frac{1}{t}}f(x)dx\Big)^r|y|_{p}^{rD+\beta}dy\Big]^{\frac{1}{r}}. \nonumber
\end{equation}
By using Minkowski's inequality again, we obtain that 
\begin{eqnarray}
\|H_{\lambda, \mu, \nu}f\|_{r, \beta}&\leq& {\bf C}_3\int_{\mathbb{Q}_p^{*}}\Big(\int_{\mathbb{Q}_p^{*}}[k_{\lambda, \mu, \nu}(x,y)]^{\frac{r}{t}}|y|_{p}^{rD+\beta}dy\Big)^\frac{1}{r}|x|_p^{\frac{\mu+\nu-\lambda}{s}-D} f(x)dx \nonumber \\
&=& {\bf C}_3\int_{\mathbb{Q}_p^{*}}\Big(\int_{\mathbb{Q}_p^{*}}\frac{|x|_p^{\frac{\mu}{t}r}|y|_p^{\frac{\nu}{t}r+rD+\beta}}{[\max\{|x|_p, |y|_{p}\}]^{\frac{\lambda}{t}r}}dy\Big)^\frac{1}{r}|x|_p^{\frac{\mu+\nu-\lambda}{s}-D} f(x)dx\nonumber \\
&=&  {\bf C}_3\int_{\mathbb{Q}_p^{*}}{{\bf I}_4}^\frac{1}{r}|x|_p^{\frac{\mu+\nu-\lambda}{s}-D} f(x)dx. \nonumber
\end{eqnarray}
Moreover, from (\ref{equi-3}) and Lemma \ref{ll-3}, we have 
\begin{eqnarray}
{{\bf I}_4}={\bf C}_4 |x|_p^{\frac{\mu}{t}r+\frac{\nu}{t}r+rD+\beta+1-\frac{\lambda}{t}r}. \nonumber 
\end{eqnarray}
Here, $$ {\bf C}_4=(1-p^{-1})\Big[1+\frac{1}{p^{\frac{\mu}{t}r+\frac{\nu}{t}r+rD+\beta+1}-1}+\frac{1}{p^{\frac{\lambda}{t}r-\frac{\mu}{t}-\frac{\nu}{t}r-rD-\beta}-1}\Big].$$
It follows that
\begin{eqnarray}
\|H_{\lambda, \mu, \nu}f\|_{r, \beta}&\leq& {\bf C}_3{\bf C}_4^{\frac{1}{r}}\int_{\mathbb{Q}_p^{*}}|x|_p^{\frac{1}{r}[\frac{\mu}{t}r+\frac{\nu}{t}r+rD+\beta+1-\frac{\lambda}{t}r]+\frac{\mu+\nu-\lambda}{s}-D} f(x)dx. \nonumber
\end{eqnarray}
Meanwhile, we note that
\begin{eqnarray}
\lefteqn{\frac{1}{r}[\frac{\mu}{t}r+\frac{\nu}{t}r+rD+\beta+1-\frac{\lambda}{t}r]+\frac{\mu+\nu-\lambda}{s}-D}\nonumber \\
&&=\mu+\nu-\lambda+\frac{\beta+1}{r}=\alpha.  \nonumber
\end{eqnarray}
This means that \begin{eqnarray}
\|H_{\lambda, \mu, \nu}f\|_{r, \beta}&\leq& {\bf C}_3{\bf C}_4^{\frac{1}{r}}\|f\|_{1, \alpha}, \nonumber
\end{eqnarray}
so that $H_{\lambda, \mu, \nu}$ is bounded from $L_{\alpha}^{q}$ to $L_{\beta}^{r}$ for $1=q\leq r<\infty$. This finishes the proof of Theorem \ref{th-1}. 

\section{{\bf Proof of Theorem \ref{th-2} and \ref{th-3}}}
\subsection{{\bf Proof of Theorem \ref{th-2}}}
To prove Theorem \ref{th-2}, we shall show the following lemma, which proves the necessity part‌ of Theorem \ref{th-2}.
\begin{lemma}
Let $\lambda, \mu, \nu$ be real numbers and $H_{\lambda, \mu, \nu}$ be as in (\ref{ope}).  If $H_{\lambda, \mu, \nu}$ is bounded from 
$L_{\alpha}^1$ to $L^{\infty}$ , then we have 
\begin{equation} 
\begin{cases}
\lambda=\mu+\nu-\alpha, \\
\mu-\lambda\leq\alpha\leq \mu, \nonumber 
\end{cases}
\end{equation}
or equivalently, 
\begin{equation} 
\begin{cases}
\lambda=\mu+\nu-\alpha, \\
0\leq\nu\leq \lambda.\nonumber
\end{cases}
\end{equation}
\end{lemma}

\begin{proof}
We suppose that $H_{\lambda, \mu, \nu}$ is bounded from
$L_{\alpha}^1$ to $L^{\infty}$.
We let $$\tau=\mu+\nu-\alpha-\lambda.$$
If $\tau<0$. For $\varepsilon>0$, we take 
 \begin{equation}f_{\varepsilon}^{[1]}(x)= 
\begin{cases}
|x|_{p}^{-\alpha-1+\varepsilon},\,\,\, {\text{when}}\, 0<|x|_p\leq 1, \\
0,\quad\quad\quad\quad\,\,\,{\text{when}}\, |x|_p>1.\nonumber 
\end{cases}
\end{equation}
Then we have $$\|f_{\varepsilon}^{[1]}\|_{1, \alpha}=\frac{1-p^{-1}}{1-p^{-\varepsilon}}.$$
We see from the boundedness of $H_{\lambda, \mu, \nu}$ that
\begin{equation}
{\bf C}_5:=\|H_{\lambda,\mu,\nu}f_{\varepsilon}^{[1]}\|_{\infty}=\sup_{y\in \mathbb{Q}_p^{*}} \int_{|x|_p\leq 1}\frac{|x|_p^{\mu}|y|_p^{\nu}}{[\max\{|x|_p, |y|_{p}\}]^{\lambda}}|x|_{p}^{-\alpha-1+\varepsilon}dx<\infty. \nonumber
\end{equation}
Then, for all $|y|_p\geq 1$, we have
\begin{eqnarray}
\int_{|x|_p\leq 1}\frac{|x|_p^{\mu}|y|_p^{\nu}}{[\max\{|x|_p, |y|_{p}\}]^{\lambda}}|x|_{p}^{-\alpha-1+\varepsilon}dx=|y|_p^{\nu-\lambda}\int_{|x|_p\leq 1}|x|_{p}^{\mu-\alpha-1+\varepsilon}dx\leq {\bf C}_5. \nonumber 
\end{eqnarray}
This implies that $\nu-\lambda\leq 0$. At the same time, for all $|y|_p\leq 1$, we have
\begin{eqnarray}
\lefteqn{{\bf C}_5\geq \int_{|x|_p\leq 1}\frac{|x|_p^{\mu}|y|_p^{\nu}}{[\max\{|x|_p, |y|_{p}\}]^{\lambda}}|x|_{p}^{-\alpha-1+\varepsilon}dx}\nonumber \\
&&\geq |y|_p^{\nu}\int_{|y|_p\leq |x|_p\leq 1}|x|_{p}^{\mu-\lambda-\alpha-1+\varepsilon}dx \nonumber \\
&&\geq  |y|_p^{\nu}\int_{|y|_p=|x|_p}|x|_{p}^{\mu-\lambda-\alpha-1+\varepsilon}dx=|y|_p^{\mu+\nu-\lambda-\alpha+\varepsilon}.\nonumber
\end{eqnarray}
This implies that $\mu+\nu-\lambda-\alpha\geq 0$. 

Meanwhile, for $\varepsilon>0$, we take 
 \begin{equation}f_{\varepsilon}^{[2]}(x)= 
\begin{cases}
0,\quad\quad\quad\quad\,\, {\text{when}}\, 0<|x|_p\leq 1, \\
|x|_{p}^{-\alpha-1-\varepsilon}, \,\,\,{\text{when}}\, |x|_p>1. \nonumber 
\end{cases}
\end{equation}
Then we have $$\|f_{\varepsilon}^{[2]}\|_{1, \alpha}=\frac{1-p^{-1}}{1-p^{-\varepsilon}}.$$
We see from the boundedness of $H_{\lambda, \mu, \nu}$ that
\begin{equation}
{\bf C}_6:=\|H_{\lambda,\mu,\nu}f_{\varepsilon}^{[2]}\|_{\infty}=\sup_{y\in \mathbb{Q}_p^{*}} \int_{|x|_p\geq 1}\frac{|x|_p^{\mu}|y|_p^{\nu}}{[\max\{|x|_p, |y|_{p}\}]^{\lambda}}|x|_{p}^{-\alpha-1-\varepsilon}dx<\infty. \nonumber 
\end{equation}
Then, for all $|y|_p\leq 1$, we have
\begin{eqnarray}
\int_{|x|_p\geq 1}\frac{|x|_p^{\mu}|y|_p^{\nu}}{[\max\{|x|_p, |y|_{p}\}]^{\lambda}}|x|_{p}^{-\alpha-1-\varepsilon}dx=|y|_p^{\nu}\int_{|x|_p\geq 1}|x|_{p}^{\mu-\lambda-\alpha-1-\varepsilon}dx\leq {\bf C}_6. \nonumber 
\end{eqnarray}
This implies that $\nu\geq 0$. Meanwhile, for all $|y|_p\geq 1$, we have
\begin{eqnarray}
\lefteqn{{\bf C}_6\geq \int_{|x|_p\geq 1}\frac{|x|_p^{\mu}|y|_p^{\nu}}{[\max\{|x|_p, |y|_{p}\}]^{\lambda}}|x|_{p}^{-\alpha-1-\varepsilon}dx}\nonumber\\
&& \geq |y|_p^{\nu-\lambda}\int_{1\leq |x|_p\leq|y|_p}|x|_{p}^{\mu-\alpha-1-\varepsilon}dx \nonumber \\
&&\geq  |y|_p^{\nu-\lambda}\int_{|y|_p=|x|_p}|x|_{p}^{\mu-\alpha-1-\varepsilon}dx=|y|_p^{\mu+\nu-\lambda-\alpha-\varepsilon}.\nonumber 
\end{eqnarray}
This implies that $\mu+\nu-\lambda-\alpha\leq 0$. Hence, combining all above arguments, we see that, if $H_{\lambda, \mu, \nu}$ is bounded from 
$L_{\alpha}^1(\mathbb{Q}_p^{*})$ to $L^{\infty}(\mathbb{Q}_p^{*})$, then we have  
\begin{equation} 
\begin{cases}
\lambda=\mu+\nu-\alpha, \\
\mu-\lambda\leq\alpha\leq \mu, \nonumber 
\end{cases}
\end{equation}
The lemma is proved. \end{proof}

We next prove that, if \begin{equation}\label{equi-6}
\begin{cases}
\lambda=\mu+\nu-\alpha, \\
\mu-\lambda\leq\alpha\leq \mu,
\end{cases}
\end{equation}
then  $H_{\lambda, \mu, \nu}$ is bounded from
$L_{\alpha}^1$ to $L^{\infty}$. First, by (\ref{equi-6}) and from Lemma \ref{ll-l}, we note that there is a constant ${\bf C}_7>0$ such that 
\begin{eqnarray}
\sup_{x,y\in \mathbb{Q}_p^{*}}\frac{|x|_p^{\mu-\alpha}|y|_p^{\nu}}{[\max\{|x|_p, |y|_{p}\}]^{\lambda}}\leq {\bf C}_7. \nonumber
\end{eqnarray}
Then, for $0\leq f\in L_{\alpha}^1$, we have
\begin{eqnarray}
\|H_{\lambda, \mu, \nu}f\|_{\infty}&=&\sup_{y\in \mathbb{Q}_p^{*}}\int_{\mathbb{Q}_p^{*}} \frac{|x|_p^{\mu-\alpha}|y|_p^{\nu}}{[\max\{|x|_p, |y|_{p}\}]^{\lambda}}f(x)|x|_{p}^{\alpha}dx\leq {\bf C}_7\|f\|_{1, \alpha}. \nonumber
\end{eqnarray}
This means that  $H_{\lambda, \mu, \nu}$ is bounded from
$L_{\alpha}^1$ to $L^{\infty}$. Theorem \ref{th-2} is proved.

\subsection{{\bf Proof of Theorem \ref{th-3}}}
We first prove the necessity part‌ of Theorem \ref{th-3}. We suppose that $H_{\lambda, \mu, \nu}$ is bounded from $L_{\alpha}^q$ to $L^{\infty}$, it follows that its adjoint operator $H^{*}_{\lambda, \mu, \nu}$ is bounded from $L^1$ to $L_{\alpha}^{q'}$. A simple computation shows that the adjoint operator $H^{*}_{\lambda, \mu, \nu}$ is given by
\begin{equation}
 H^{*}_{\lambda, \mu, \nu}f(y)=\int_{\mathbb{Q}_p^{*}}\frac{|x|_p^{\nu}|y|_p^{\mu-\alpha}}{[\max\{|x|_p, |y|_{p}\}]^{\lambda}}f(x)dx. \nonumber 
\end{equation}
It follows from Lemma \ref{ll-4} that the boundedness of $H^{*}_{\lambda, \mu, \nu}: L^1 \rightarrow L_{\alpha}^{q'}$ implies that 
\begin{equation}\label{equi-5} 
\begin{cases}
\lambda=\mu+\nu-\alpha+\frac{\alpha+1}{q'}, \\
-q'(\mu-\alpha)<\alpha+1<q'(\lambda-\mu+\alpha).
\end{cases}
\end{equation}
That is \begin{equation}
\begin{cases}
\lambda=\mu+\nu+1-\frac{\alpha+1}{q}, \\
0<\nu<\lambda,\nonumber
\end{cases}
\end{equation}or equivalently, 
\begin{equation}
\begin{cases}
\lambda=\mu+\nu+1-\frac{\alpha+1}{q}, \\
q(\mu+1-\lambda)<\alpha+1<q(\mu+1). \nonumber
\end{cases}
\end{equation}
Then the necessity part‌ of Theorem \ref{th-3} is proved. 

We next prove the sufficiency part of Theorem \ref{th-3}. We will show that, if (\ref{equi-5}) holds, then $H_{\lambda, \mu, \nu}$ is bounded from $L_{\alpha}^q$ to $L^{\infty}$. When $-q'(\mu-\alpha)<\alpha+1<q'(\lambda-\mu+\alpha)$, that is 
$$q'(\mu-\alpha)+\alpha+1>0,\,\, {\text{and}}\,\, q'\lambda-q'(\mu-\alpha)-\alpha-1>0.$$ We see from Lemma \ref{ll-3} and $\lambda=\mu+\nu-\alpha+\frac{\alpha+1}{q'}$ that
\begin{eqnarray}
\lefteqn{\int_{\mathbb{Q}_p^{*}} \frac{|x|_p^{q'(\mu-\alpha)+\alpha}|y|_p^{q'\nu}}{[\max\{|x|_p, |y|_{p}\}]^{q'\lambda}}f(x)dx}\nonumber \\
&&=(1-p^{-1})\Big[1+\frac{1}{p^{q'(\mu-\alpha)+\alpha+1}-1}+\frac{1}{p^{q'\lambda-q'(\mu-\alpha)-\alpha-1}-1}\Big]:={\bf C}_8.  \nonumber
\end{eqnarray}
Consequently, by using H\"older's inequality, we obtain that
\begin{eqnarray}
\|H_{\lambda, \mu, \nu}f\|_{\infty}&=&\sup_{y\in \mathbb{Q}_p^{*}}\int_{\mathbb{Q}_p^{*}} \frac{|x|_p^{\mu}|y|_p^{\nu}}{[\max\{|x|_p, |y|_{p}\}]^{\lambda}}f(x)dx \nonumber \\
&\leq & \sup_{y\in \mathbb{Q}_p^{*}} \|f\|_{q, \alpha}\Big[\int_{\mathbb{Q}_p^{*}} \frac{|x|_p^{q'(\mu-\alpha)+\alpha}|y|_p^{q'\nu}}{[\max\{|x|_p, |y|_{p}\}]^{q'\lambda}}dx\Big]^{\frac{1}{q'}}={\bf C}_8 \|f\|_{q, \alpha}.  \nonumber
\end{eqnarray}
This means that $H_{\lambda, \mu, \nu}$ is bounded from $L_{\alpha}^q$ to $L^{\infty}$. This proves Theorem \ref{th-3}.

\section{{\bf Proof of Theorem \ref{th-4}}}
We only need to prove that, when $H_{\lambda, \mu, \nu}$ is bounded from $L_{\alpha}^{q}$ to $L_{\beta}^{q}$, the norm of $H_{\lambda, \mu, \nu}$ is given by
 $$\|H_{\lambda, \mu, \nu}\|_{L_{\alpha}^{q}\rightarrow L_{\beta}^q}=(1-p^{-1})\Big[1+\frac{1}{p^{\mu+1-\frac{1}{q}(\alpha+1)}-1}+\frac{1}{p^{\nu+\frac{1}{q}(\beta+1)}-1}\Big],\,\, 1\leq q<\infty,$$ 
because other parts have been proved in Theorem \ref{th-1}. 

We first consider $q=1$. When $H_{\lambda, \mu, \nu}$ is bounded from $L_{\alpha}^{q}$ to $L_{\beta}^{q}$, it holds that
\begin{equation} 
\begin{cases}
\lambda=\mu+\nu+1+\beta-\alpha, \\
-\nu<\beta+1<\lambda-\nu, \nonumber
\end{cases}
\end{equation} From Lemma \ref{ll-1}, we know that 
\begin{eqnarray}
\lefteqn{\|H_{\lambda, \mu, \nu}\|_{L_{\alpha}^{1}\rightarrow L_{\beta}^1}=\sup_{x\in \mathbb{Q}_p^{*}}\int_{\mathbb{Q}_p^{*}}\frac{|x|_p^{\mu-\alpha}|y|_{p}^{\nu+\beta}}{[\max\{|x|_p, |y|_p\}]^{\lambda}}dy}\nonumber \\
&&=(1-p^{-1})\Big[1+\frac{1}{p^{\nu+\beta+1}-1}+\frac{1}{p^{\lambda-\nu-\beta-1}-1}\Big]\nonumber \\
&&=(1-p^{-1})\Big[1+\frac{1}{p^{\mu-\alpha}-1}+\frac{1}{p^{\nu+\beta+1}-1}\Big]. \nonumber 
\end{eqnarray}
This proves the case $q=1$ of Theorem \ref{th-4}.

Next, we consider the case $q>1$. We borrow some arguments of the proof of {\bf Case ${\bf I}$} of Theorem \ref{th-1}. Actually, when $q=r$ and
\begin{equation} 
\begin{cases}
\lambda=\mu+\nu+1+\frac{\beta-\alpha}{q}, \\
q(\mu+1-\lambda)<\alpha+1<q(\mu+1), \nonumber 
\end{cases}
\end{equation}
we can take $s=q', t=q$ and $A=-\frac{\alpha+1}{qq'}$, in the proof of {\bf Case ${\bf I}$} of Theorem \ref{th-1}, so that (\ref{equi-1}) and (\ref{equi-2}) are satified. Then, after repeating the arguments in the proof of {\bf Case ${\bf I}$} of Theorem \ref{th-1}, we can obtain that
$${\bf C}_1={\bf C}_2=(1-p^{-1})\Big[1+\frac{1}{p^{\mu+1-\frac{1}{q}(\alpha+1)}-1}+\frac{1}{p^{\nu+\frac{1}{q}(\beta+1)}-1}\Big].$$
and
$$\|H_{\lambda, \mu, \nu}\|_{L_{\alpha}^{q}\rightarrow L_{\beta}^q}\leq (1-p^{-1})\Big[1+\frac{1}{p^{\mu+1-\frac{1}{q}(\alpha+1)}-1}+\frac{1}{p^{\nu+\frac{1}{q}(\beta+1)}-1}\Big].$$

We finally prove that  $$\|H_{\lambda, \mu, \nu}\|_{L_{\alpha}^{q}\rightarrow L_{\beta}^q}\geq (1-p^{-1})\Big[1+\frac{1}{p^{\mu+1-\frac{1}{q}(\alpha+1)}-1}+\frac{1}{p^{\nu+\frac{1}{q}(\beta+1)}-1}\Big].$$ 
For $\varepsilon>0$, we take 
 \begin{equation}\label{r-4}f_{\varepsilon}^{}(x)= 
\begin{cases}
0,\quad\quad\quad\quad\quad\,\,\,\,\, {\text{when}}\, 0<|x|_p<1, \\
c_{p,\varepsilon}|x|_{p}^{-\frac{\alpha+1}{q}-\frac{\varepsilon}{q}},\,\,{\text{when}}\, |x|_p\geq1.
\end{cases}
\end{equation}
Here $$c_{p,\varepsilon}=\Big(\frac{1-p^{-\varepsilon}}{1-p^{-1}}\Big)^{\frac{1}{q}}.$$
Then we have $\|f_{\varepsilon}\|_{q, \alpha}=1$ and 
\begin{eqnarray}
\|H_{\lambda, \mu, \nu}f_{\varepsilon}\|_{q, \beta}&=&c_{p,\varepsilon}\Big[\int_{\mathbb{Q}_p^{*}}|y|_p^{\beta}\Big(\int_{|x|_p\geq 1}\frac{|x|_p^{\mu}|y|_{p}^{\nu}}{[\max\{|x|_p, |y|_p\}]^{\lambda}}|x|_{p}^{-\frac{\alpha+1}{q}-\frac{\varepsilon}{q}}dx\Big)^qdy\Big]^{\frac{1}{q}} \nonumber 
\\
&=& c_{p,\varepsilon}\Big[\int_{\mathbb{Q}_p^{*}}|y|_p^{-1-\varepsilon}\Big(\int_{|t|_p\geq \frac{1}{|y|_p}}\frac{|t|_p^{\mu-\frac{\alpha+1}{q}-\frac{\varepsilon}{q}}}{[\max\{1, |t|_p\}]^{\lambda}}dt\Big)^qdy\Big]^{\frac{1}{q}}. \nonumber 
\end{eqnarray}

Now, we take $\varepsilon=p^{-\bf N}$ so that $|\varepsilon|_p=p^{{\bf N}}$, ${\bf N} \in \mathbb{N}$. Then we have 
\begin{eqnarray}
\|H_{\lambda, \mu, \nu}f_{\varepsilon}\|_{q, \beta}&\geq & c_{p,\varepsilon}\Big[\int_{|y|_p \geq |\varepsilon|_p}|y|_p^{-1-\varepsilon}\Big(\int_{|t|_p\geq \frac{1}{|\varepsilon|_p}}\frac{|t|_p^{\mu-\frac{\alpha+1}{q}-\frac{\varepsilon}{q}}}{[\max\{1, |t|_p\}]^{\lambda}}dt\Big)^qdy\Big]^{\frac{1}{q}}. \nonumber 
\end{eqnarray}
It follows that
\begin{eqnarray}\label{ms} \|H_{\lambda, \mu, \nu}\|_{L_{\alpha}^q \rightarrow L_{\beta}^q} \geq \|H_{\lambda, \mu, \nu}f_\varepsilon\|_{q, \beta}\geq(\varepsilon^{\varepsilon})^{\frac{1}{q}}
\int_{|t|_p \geq \frac{1}{|\varepsilon|_p}}\frac{|t|_p^{\mu-\frac{\alpha+1}{q}-\frac{\varepsilon}{q}}}{[\max\{1, |t|_p\}]^{\lambda}}dt.  \end{eqnarray}
We let 
$$\mathbf{E}_{{\bf N}}=\{t\in \mathbb{Q}_p^{*}: |t|_p \geq \frac{1}{|\varepsilon|_p}\}=\{t\in \mathbb{Q}_p^{*}: |t|_p \geq \frac{1}{p^{\bf N}}\}.$$
Then 
\begin{eqnarray}\int_{|t|_p \geq \frac{1}{|\varepsilon|_p}}\frac{|t|_p^{\mu-\frac{\alpha+1}{q}-\frac{\varepsilon}{q}}}{[\max\{1, |t|_p\}]^{\lambda}}dt=\int_{\mathbb{Q}_p^{*}}\frac{\chi_{\mathbf{E}_N}(t)
|t|_{p}^{\mu-\frac{\alpha+1}{q}-\frac{1}{q}p^{-\bf N}}}{[\max\{1, |t|_p\}]^{\lambda}}dt.\nonumber
\end{eqnarray}
Note that, for each $t\in \mathbb{Q}_p^{*}$,
\begin{eqnarray}\frac{\chi_{\mathbf{E}_N}(t)
|t|_{p}^{\mu-\frac{\alpha+1}{q}-\frac{1}{q}p^{-\bf N}}}{[\max\{1, |t|_p\}]^{\lambda}}\rightarrow \frac{|t|_{p}^{\mu-\frac{\alpha+1}{q}}}{[\max\{1, |t|_p\}]^{\lambda}},\,\, {\text as}\,\, N\rightarrow \infty,\nonumber
\end{eqnarray}
and $(\varepsilon^{\varepsilon})^{\frac{1}{q}}\rightarrow 1$ as $N\rightarrow \infty.$ Consequently, by Fatou's lemma, we see from (\ref{ms}) that 
\begin{eqnarray*} \|H_{\lambda, \mu, \nu}\|_{L_{\alpha}^q \rightarrow L_{\beta}^q} \geq
\int_{\mathbb{Q}_{p}^{*}}\frac{|t|_{p}^{\mu-\frac{\alpha+1}{q}}}{[\max\{1, |t|_p\}]^{\lambda}}\,dt ={\bf C}_{1}. \end{eqnarray*}
This proves the case $q>1$ of Theorem \ref{th-4} and the proof of Theorem \ref{th-4} is finished.
 
\section{{\bf Proof of Theorem \ref{th-5}}}

From Lemma \ref{ll-2}, we see that $H_{\lambda, \mu, \nu}$ is bounded on $L^{\infty}$ if and only if 
\begin{equation}\label{inf}\sup_{x\in \mathbb{Q}_p^{*}}\int_{\mathbb{Q}_p^{*}}\frac{|x|_p^{\mu}|y|_p^{\nu}}{[\max\{|x|_p, |y|_{p}\}]^{\lambda}}dy<\infty.\end{equation}
Also, it follows from Lemma \ref{ll-3} that (\ref{inf}) holds if and only if 
\begin{equation} 
\begin{cases}
\nu+1>0, \\
\lambda-\nu-1>0, \\
\mu+\nu+1-\lambda=0.\nonumber 
\end{cases}
\end{equation}
That is 
\begin{equation} 
\begin{cases}
\lambda=\mu+\nu+1, \\
0<\nu+1<\lambda,\nonumber 
\end{cases}
\end{equation}
or equivalently, 
\begin{equation} 
\begin{cases}
\lambda=\mu+\nu+1, \\
0<\mu<\lambda.\nonumber 
\end{cases}
\end{equation}
Moreover, when $H_{\lambda, \mu, \nu}$ is bounded on $L^{\infty}$, we see from again Lemma \ref{ll-2} and Lemma \ref{ll-3} that the norm of $H_{\lambda, \mu, \nu}$ is given by
 $$\|H_{\lambda, \mu, \nu}\|_{L^{\infty}\rightarrow L^{\infty}}=(1-p^{-1})\Big[1+\frac{1}{p^{\nu+1}-1}+\frac{1}{p^{\lambda-\nu-1}-1}\Big].$$ 
This proves Theorem \ref{th-5}.

\section{\bf{Final remarks}}
In this section, we deal with the $L^q\rightarrow L^{r}$ boundedness of $H_{\lambda, \mu, \nu}$ for the remaining cases $1\leq r<q\leq\infty.$ We will show that 
$H_{\lambda, \mu, \nu}$ is unbounded for all the cases $1\leq r<q\leq\infty.$ 

\begin{remark}\label{r-l}
We first consider the case $1=r<q<\infty$. In view of Lemma \ref{ll-1}, we know that $H_{\lambda, \mu, \nu}$ is bounded from $L_{\alpha}^{q}$ to $L_{\beta}^1$ if and only if 
\begin{equation}\label{r-1}\mathcal{I}_1(x):=\int_{\mathbb{Q}_p^{*}}\frac{|x|_p^{\mu-\frac{\alpha}{q}}|y|_p^{\nu+\beta}}{[\max\{|x|_p, |y|_{p}\}]^{\lambda}}dy\in L^{q'}.\end{equation}
Furthermore, from Lemma \ref{ll-3}, we know that, when (\ref{r-1}) holds, we have \begin{equation}\label{r-2}\nu+\beta+1>0,\,\,{\text{and}}\,\, \lambda-\nu-\beta-1>0.\end{equation}  
However, when (\ref{r-2}) holds, we have  
\begin{equation}
 \mathcal{I}_1(x)=|x|_p^{\mu+\nu-\frac{\alpha}{q}+\beta+1-\lambda}\Big[1+\frac{1}{p^{\nu+\beta+1}-1}+\frac{1}{p^{\lambda-\nu-\beta-1}-1}\Big], \nonumber
\end{equation}
It follows that $\|\mathcal{I}_1\|_{q'}=\infty$. This means that $H_{\lambda, \mu, \nu}$ is unbounded from $L_{\alpha}^{q}$ to $L_{\beta}^1$ for $q>1$. 
\end{remark}
\begin{remark}
We then consider the case $1\leq r<q=\infty$. When $1=r<q=\infty$, from Lemma \ref{ll-2}, we know that $H_{\lambda, \mu, \nu}: L^{\infty} \rightarrow L_{\beta}^{1}$ is bounded if and only if
\begin{equation}\label{r-3-1}\mathcal{I}_2(y):=\int_{\mathbb{Q}_p^{*}}\frac{|x|_p^{\mu}|y|_p^{\nu+\beta}}{[\max\{|x|_p, |y|_{p}\}]^{\lambda}}dx\in L^{1}.\end{equation}
Meanwhile, fom Lemma \ref{ll-3}, we know that, when (\ref{r-3-1}) holds, we have \begin{equation}\label{r-4-1}\mu+1>0,\,\,{\text{and}}\,\, \lambda-\mu-1>0.\end{equation}  
But, when (\ref{r-4-1}) holds, we have  
\begin{equation}
 \mathcal{I}_2(y)=|y|_p^{\nu+\beta+\mu+1-\lambda}\Big[1+\frac{1}{p^{\mu+1}-1}+\frac{1}{p^{\lambda-\mu-1}-1}\Big], \nonumber
\end{equation}
so that $\|\mathcal{I}_2\|_{1}=\infty$. This means that $H_{\lambda, \mu, \nu}$ is unbounded from $L^{\infty}$ to $L_{\beta}^1$.  When $1<r<q=\infty$, we consider the adjoint operator of $H_{\lambda, \mu, \nu}$ and use similar arguments in Remark \ref{r-l}, we can obtain that $H_{\lambda, \mu, \nu}: L^{\infty} \rightarrow L_{\beta}^{r}$ is unbounded for this case. We omit the details of the proof for this claim. 
\end{remark}

\begin{remark}
We now consider the case $1<r<q<\infty$. In view of Lemma \ref{ll-4}, we know that, when $H_{\lambda, \mu, \nu}$ is bounded from $L_{\alpha}^{q}$ to $L_{\beta}^{r}$, it holds that
 \begin{equation}\label{r-3}
\begin{cases}
\lambda=\mu+\nu+1+\frac{\beta+1}{r}-\frac{\alpha+1}{q}, \\
-r\nu<\beta+1<r(\lambda-\nu),
\end{cases}
\end{equation}
But, when (\ref{r-3}) holds, for $\varepsilon>0$, take $f_{\varepsilon}$ as in (\ref{r-4}), we have $\|f_{\varepsilon}\|_{q, \alpha}=1$ and 
\begin{eqnarray}
\|H_{\lambda, \mu, \nu}f_{\varepsilon}\|_{r, \beta}&=&c_{p,\varepsilon}\Big[\int_{\mathbb{Q}_p^{*}}|y|_p^{\beta}\Big(\int_{|x|_p\geq 1}\frac{|x|_p^{\mu}|y|_{p}^{\nu}}{[\max\{|x|_p, |y|_p\}]^{\lambda}}|x|_{p}^{-\frac{\alpha+1}{q}-\frac{\varepsilon}{q}}dx\Big)^rdy\Big]^{\frac{1}{r}} \nonumber 
\\
&=& c_{p,\varepsilon}\Big[\int_{\mathbb{Q}_p^{*}}|y|_p^{-1-\frac{\varepsilon}{q}r}\Big(\int_{|t|_p\geq \frac{1}{|y|_p}}\frac{|t|_p^{\mu-\frac{\alpha+1}{q}-\frac{\varepsilon}{q}}}{[\max\{1, |t|_p\}]^{\lambda}}dt\Big)^rdy\Big]^{\frac{1}{r}} \nonumber \\
&\geq &  c_{p,\varepsilon}\Big[\int_{|y|_p\geq 1}|y|_p^{-1-\frac{\varepsilon}{q}r}dy\Big]^{\frac{1}{r}}\Big(\int_{|t|_p\geq 1}\frac{|t|_p^{\mu-\frac{\alpha+1}{q}-\frac{\varepsilon}{q}}}{[\max\{1, |t|_p\}]^{\lambda}}dt\Big). \nonumber
\end{eqnarray}
Moreover, we have
\begin{eqnarray}
 c_{p,\varepsilon}\Big[\int_{|y|_p\geq 1}|y|_p^{-1-\frac{\varepsilon}{q}r}dy\Big]^{\frac{1}{r}}=(1-p^{-1})^{\frac{1}{r}-\frac{1}{q}}\frac{(p^{\varepsilon}-1)^{\frac{1}{q}}}{(p^{\frac{r}{q}\varepsilon}-1)^{\frac{1}{r}}}\rightarrow \infty, \nonumber 
\end{eqnarray}
as $\varepsilon \rightarrow 0^{+}$ since $1<r<q$. This implies that $H_{\lambda, \mu, \nu}$ is unbounded from $L_{\alpha}^{q}$ to $L_{\beta}^{r}$ for $1<r<q<\infty$.
\end{remark}


\begin{thebibliography}{99}

 \bibitem{BS}Bansah J., Sehba B., {\em Boundedness of a family of Hilbert-type operators and its Bergman-type analogue}, Illinois J. Math., 59 (2015), no. 4, pp. 949-977.
 
\bibitem{BB}Behera B., {\em Hardy and Hardy-Littlewood-Pólya operators and their commutators on local fields}, Period. Math. Hungar., 89 (2024), no. 2, pp. 318-334.


\bibitem{B}Brevig O., {\em Sharp norm estimates for composition operators and Hilbert-type inequalities}, Bull. Lond. Math. Soc., 49(2017), pp.965-978.

\bibitem{B-1}Brevig O.,{\em The best constant in a Hilbert-type inequality}, Expo. Math.,  42 (2024), no.1, Paper No. 125530, 11 pp.
 
\bibitem{DD1}Dung K.,  Duong D., {\em The $p$-adic Hausdorff operator and some applications to Hardy-Hilbert type inequalities},  Russ. J. Math. Phys., 28 (2021), no. 3, pp. 303-316.
\bibitem{DD2}Dung K.,  Duong D., {\em Weighted Triebel-Lizorkin and Herz spaces estimates for $p$-adic Hausdorff type operator and its applications},  Anal. Math., 48(2022), no.3, pp. 717-740.
\bibitem{DD3}Dung K.,  Duong D., {\em Two-weight estimates for Hardy-Littlewood maximal functions and Hausdorff operators on  $p$-adic Herz spaces},  Izv. Math., 87 (2023), no. 5, pp. 920-940.
 \bibitem{FWL}Fu Z., Wu Q., Lu S., {\em Sharp estimates of p-adic Hardy and Hardy-Littlewood-P\'{o}lya operators}, Acta Mathematica Sinica, Vol. 29, no. 1, pp. 137-150, 2013.
\bibitem{HLP}Hardy  G., Littlewood J., P\'{o}lya G., {\em Inequalities}, Cambridge University Press, Cambridge, 1952.
\bibitem{J-1}Li H., Jin J., {\em On a $p$-adic Hilbert-type integral operator and its applications}, J. Appl. Anal. Comput., 10 (2020), no. 4, pp. 1326-1334.
\bibitem{J-2} Jin J., {\em Some new $p$-adic Hardy-Littlewood-P\'{o}lya-type inequalities}, Acta Math. Sinica (Chinese Ser.),  63 (2020), no. 6, 639-646.
\bibitem{J-5} Jin J., {\em On the operators of Hardy-Littlewood-Pólya type}, J. Appl. Anal. Comput., 15 (2025), no. 1, 470-487.
\bibitem{J-3} Jin J., Tang S., and Feng X., {\em On a $p$-adic integral operator induced by a homogeneous kernel and its applications}, Acta Math. Sci. Ser. A (Chinese Ed.),  41 (2021), no. 4, 968-977.
\bibitem{J-4}Jin J., {\em A new class of p-adic integral operators and applications}, $p$-Adic Numbers Ultrametric Anal. Appl., 16 (2024), no. 4, pp. 390-400.

\bibitem{O}Okikiolu G., {\em On inequalities for integral operators}, Glasg. Math. J., 11(1970), pp. 126-133.
\bibitem{T}Taibleson M., {\em Fourier Analysis on Local Fields}, Princeton University Press, Princeton, 1975.
\bibitem{Tao}Tao T., {\em Harmonic analysis}, Lecture notes at UCLA, \\http://www.math.ucla.edu/~tao/247a.1.06f/notes2.pdf.
\bibitem{VVZ}Vladimirov V., Volovich I., Zelenov E., {\em $p$-Adic Analysis and Mathematical Physics}, World Scientific, Singapore, 1994.

\bibitem{WHY}Wu F., Hong Y., Yang B., {\em A refined Hardy-Littlewood-Polya inequality and the equivalent forms}, J. Math. Inequal., 16(2022), no.4, pp. 1477-1491. 

\bibitem{WF}Wu Q., Fu Z., {\em Sharp estimates of m-linear $p$-adic Hardy and Hardy-Littlewood-P\'{o}lya operators}, J. Appl. Math., 2011, Art. ID 472176, 20 pp.
 
\bibitem{YR}Yang B., Rassias T., {\em On the way of weight coefficient and research for the Hilbert-type inequalities}, Math. Inequal. Appl., 6(2003), pp. 625-658.

\bibitem{Y3}Yang B., {\em The Norm of Operator and Hibert-type Inequalities(in Chinese)}, Science Press, 2009.

\bibitem{YZ}Yang, B., Zhong Y., {\em On a reverse Hardy-Littlewood-P\'{o}lya's inequality}, J. Appl. Anal. Comput., 10(2020), no.5, pp. 2220-2232.


\bibitem{Zh}Zhao R., {\em Generalization of Schur’s test and its application to a class of integral operators on the unit ball of $\mathbb{C}^n$}, Integral Equ. Oper. Theory, 82(2015), pp. 519-532.

\end{thebibliography}
\end{document}